\documentclass[a4paper,11pt]{amsart}
\usepackage{amsmath,amssymb,amsthm,mathtools,amscd,mathrsfs}
\usepackage{indentfirst}
\usepackage{stmaryrd}
\usepackage{graphicx}
\usepackage{extarrows}
\usepackage{romanbar}
\usepackage{multirow}
\usepackage{amsfonts}
\usepackage{color}
\usepackage{pictexwd,dcpic}
\usepackage{psfrag}
\usepackage{hyperref}
\usepackage{comment}
\usepackage{enumerate}

\usepackage{tabularx,array,float,microtype}
\allowdisplaybreaks[4]

\definecolor{DarkBlue}{rgb}{0.1,0.1,0.55}
\definecolor{DarkRed}{rgb}{0.55,0.1,0.1}

\hypersetup{
  colorlinks=true,
  linkcolor=DarkBlue,
  citecolor=DarkRed,
  urlcolor=DarkBlue,
  pdftitle={On Chern's Conjecture for Minimal Submanifolds with Flat Normal Bundle in Spheres},
  pdfauthor={Jianquan Ge, Fagui Li and Yunheng Zhang}
}

\numberwithin{equation}{section}
\theoremstyle{plain}
\newtheorem{thm}{Theorem}[section]

\newtheorem{prop}[thm]{Proposition}

\newtheorem{lem}[thm]{Lemma}
\newtheorem{cor}[thm]{Corollary}

\newtheorem{rem}[thm]{Remark}

\newtheorem*{simonstheorem}{Simons' theorem}
\newtheorem*{Chernconjecture}{Chern's conjecture}
\newtheorem*{PengTerngtheorem}{Peng--Terng's theorem}
\theoremstyle{definition}

\makeatletter
\@namedef{subjclassname@2020}{\textup{2020} Mathematics Subject Classification}
\renewcommand{\@defaultbiblabelstyle}[1]{[#1]}
\renewcommand\subsection{\@startsection{subsection}{2}{\z@}%
  {2.0ex plus 0.8ex minus 0.2ex}% beforeskip
  {1.0ex plus 0.2ex}% afterskip (positive => text starts on next line)
  {\normalfont\bfseries}}
\makeatother

\title[Chern's conjecture for submanifolds with flat normal bundle
]{On Chern's Conjecture for Minimal Submanifolds with Flat Normal Bundle in Spheres}

\author[J. Q. Ge]{Jianquan Ge}
\address{School of Mathematical Sciences, Laboratory of Mathematics and Complex Systems, Beijing Normal University, Beijing 100875, P. R. China.}
\email{jqge@bnu.edu.cn}

\author[F. G. Li]{Fagui Li$^{*}$}
\address{$^{*}$Frontier Interdisciplinary Domain, Beijing Institute of Technology, Zhuhai, Guangdong 519088, P. R. China.}
\email{lifagui@bitzh.edu.cn}

\author[Y. H. Zhang]{Yunheng Zhang}
\address{School of Mathematical Sciences, Laboratory of Mathematics and Complex Systems, Beijing Normal University, Beijing 100875, P. R. China.}
\email{yunheng@mail.bnu.edu.cn}
\date{}
\subjclass[2020]{53C20, 53C24, 53C42}
\keywords{Chern's conjecture, minimal submanifolds, rigidity theorem, flat normal bundle.}
\thanks{$^{*}$ Corresponding author.}
\thanks{J. Q. Ge is partially supported by NSFC (No. 12571049) and the Fundamental Research Funds for the Central Universities.}
\thanks{F. G. Li is partially supported by NSFC (No. 12271040 and 12501061), the Guangdong Provincial Association for Science and Technology Youth Talent Support Program (No. SKXRC2026413) and the Research Start-up Funding of Beijing Institute of Technology (No. 5640011253301).}
\begin{document}
	\begin{abstract}
Let $M^n$ $(n\geqslant3)$ be a closed minimal submanifold in the unit sphere $\mathbb S^{n+m}$ $(m\geqslant2)$ with flat normal bundle, and let $S$ denote the squared norm of its second fundamental form. We prove an explicit second-gap rigidity theorem for   $S$. More precisely, if $S$ is constant and
\[
0\leqslant S\leqslant n+\delta,
\]
where $\delta$ is an explicit constant satisfying $\delta\geqslant \frac{n}{87}$, then either $S\equiv0$ and $M$ is a totally geodesic sphere, or $S\equiv n$ and $M$ is a Clifford torus contained in a totally geodesic $\mathbb S^{n+1}\subset\mathbb S^{n+m}$. 
%We observe that the flat-normal-bundle assumption is necessary here. 
The flat-normal-bundle condition is essential in the general higher-codimensional setting: without it, the corresponding rigidity statement already fails in dimension two.
This theorem provides positive evidence for Chern's conjecture in higher codimension.
\end{abstract}
	\maketitle
	
	\section{Introduction}
	In 1968, Simons \cite{Simons} established the celebrated rigidity theorem stated below:
	
%	\vspace{2mm}
%	\noindent
%	\textbf{Simons theorem.}
%	\textit{
\begin{simonstheorem}
		Let $M^n$ be a closed, minimally immersed submanifold in the unit
		sphere $\mathbb{S}^{n+m}$, and let $S$ denote the squared norm of its second
		fundamental form. Then
		\[\int_M S\left(S-\frac{n}{2-\frac{1}{m}}\right) dM\geqslant0.\]
		In particular, for $S\leqslant\frac{n}{2-\frac{1}{m}}$
		one has either $S\equiv0$ or  $S\equiv\frac{n}{2-\frac{1}{m}}$
		on $M^n$.
\end{simonstheorem}
%	\vspace{2mm}
	
	This theorem indicates that if $S$ is constant, it cannot take any
	value in the open interval $(0, \frac{n}{2-\frac{1}{m}})$. By systematically employing the moving frame method, Chern \cite{chern} recovered Simons' inequality above. Furthermore, in collaboration with do Carmo and Kobayashi (cf.\ \cite{CCK}), he showed that the only submanifolds on which
	$S\equiv\frac{n}{2-\frac{1}{m}}$ are the Clifford
	minimal hypersurfaces and the Veronese surface (in $\mathbb{S}^4$). Afterwards, Li--Li \cite{1992} and Chen--Xu \cite{CX93} improved Simons' pinching constant to $\frac{2n}{3}$ for codimension $m\geqslant2$.
	
	Inspired by this discovery, the renowned Chern conjecture was proposed by Chern in \cite{chern} and by Chern, do Carmo and Kobayashi in \cite{CCK}, in 1968 and 1970, respectively.
	
	%\vspace{2mm}
%	\noindent
%	\textbf{Chern conjecture.}
	%\textit{
	\begin{Chernconjecture}
	Let $M^n$ be a closed, minimally immersed submanifold in the unit
		sphere $\mathbb{S}^{n+m}$ with constant scalar curvature, equivalently with constant length of the second fundamental form, whose squared norm is denoted by $S$. Then for each $n$, the set of all possible values for $S$ is discrete.
		\end{Chernconjecture}%}
%	\vspace{2mm}
	
	Motivated by isoparametric theory and by the fact that
	isoparametric hypersurfaces are the only known examples of minimal hypersurfaces with
	constant scalar curvature in $\mathbb{S}^{n+1}$ (see Chi \cite{Chi21} for a historical survey),
	the Chern conjecture in the hypersurface case has been reformulated as the following stronger version (cf.\ \cite{Verstraelen}); see also \cite{ScherfnerWeissYau12,XuXuSurvey24} for surveys of this problem and its isoparametric formulation.
	
%	\vspace{2mm}
%	\noindent
%	\textbf{Chern conjecture (Stronger version).}
%	\textit{
	\begin{Chernconjecture}[Stronger version]
		Let $M^n$ be a closed, minimally immersed hypersurface of the unit
		sphere $\mathbb{S}^{n+1}$ with constant scalar curvature. Then $M^n$ is isoparametric.
			\end{Chernconjecture}
%	}
%	\vspace{2mm}
	
	Since then, a long sequence of improvements has been obtained in the hypersurface case. In 1983, Peng--Terng \cite{Peng-T1, Peng-T2} pioneered the study of the second pinching problem for minimal hypersurfaces in the unit sphere and made the following breakthrough regarding the Chern conjecture. 
	
%	\vspace{2mm}
%	\noindent
%	\textbf{Peng--Terng theorem.}
		\begin{PengTerngtheorem}
%	\textit
	Let $M^n$ be a closed minimal hypersurface in $\mathbb{S}^{n+1}$.
	\begin{enumerate}
		\item If $S$ is constant and $n\leqslant S\leqslant n+\frac{1}{12n}$, then $S=n$.
		\item If $n\leqslant5$ and $n\leqslant S\leqslant n+\frac{6-1.13n}{5+\sqrt{17}}$, then $S\equiv n$.
\end{enumerate}
	\end{PengTerngtheorem}
	%}
%\vspace{2mm}

Thereafter, numerous results concerning the second gap of $S$ have emerged. Chang \cite{C2} proved that \emph{a closed minimally immersed hypersurface with constant scalar curvature in $\mathbb{S}^{4}$ is isoparametric}, which solved the stronger version of the Chern conjecture in the case $n=3$. When $S$ is constant, Yang--Cheng \cite{YC} improved the pinching constant from $\frac{1}{12n}$ to $\frac{n}{3}$ in 1998. Later, Suh--Yang \cite{2007} further improved it to $\frac{3n}{7}$.
 Without assuming the constancy of $S$, Cheng--Ishikawa \cite{1999}, Wei--Xu \cite{WX}, and Zhang \cite{Zhang} partially improved the second gap in low dimensions. Ding--Xin \cite{DX} first extended it to arbitrary dimensions. Building on further refinements by Xu--Xu \cite{2016}, Lei--Xu--Xu \cite{2017} improved the second gap to $\frac{n}{18}$. 
 
 Another important direction assumes additional trace information, such as the constancy of $f_3=\sum\limits_i\lambda_i^3$, under which stronger conclusions are available. Tang--Wei--Yan \cite{TangWeiYan20} and Tang--Yan \cite{TangYan23} obtained strong sufficient conditions ensuring that hypersurfaces with prescribed higher mean curvatures or trace invariants are isoparametric. These results extend \cite{A-B} to arbitrary dimensions. Cheng--Wei--Yamashiro \cite{ChengWeiYamashiro25} showed that \emph{for a complete minimal hypersurface with constant scalar curvature, if $f_{3}$ is constant and $S>n$, then $S>1.8252n-0.712898$}. Related rigidity and second-gap results in dimension four and in Willmore-type settings were obtained by Deng--Gu--Wei \cite{DengGuWei17}, Li \cite{Li22}, Ge--Tan--Yan--Zhang \cite{GeTanYanZhang25}, He--Xu--Zhao \cite{HeXuZhao26}, Tao \cite{Tao26}, Ge--Liu--Luo--Yan \cite{GLLY} and Deng-Kou \cite{DengKou2026}, etc. These results provide strong evidence for the expected discreteness and rigidity characterization in codimension one.
 
 However, very little is known about the higher-codimensional situation. In 2011, Lu \cite{Lu11} proposed a natural higher-codimensional analogue involving the second largest eigenvalue of the fundamental matrix
 \[
 \mathcal A=(\langle A^\alpha,A^\beta\rangle)_{m\times m}.
 \]
 Let $a_{1}\geqslant a_{2}\geqslant\cdots\geqslant a_{m}$ denote the eigenvalues of the positive semidefinite symmetric matrix $\mathcal{A}$.
 For related pinching-rigidity results for minimal surfaces in spheres, see Ding--Ge--Li \cite{DingGeLi25}. For surfaces, Ding--Ge--Li--Yang \cite{DingGeLiYang26} recently proved Lu's conjecture completely for 2-spheres and for general surfaces under mild assumptions on the normal scalar curvature. In particular, they proved the following consequence:
\begin{quote}
\itshape\noindent
For a closed minimal surface $M^{2}$ immersed in the unit sphere $\mathbb{S}^{2+m}$, if the normal bundle of $M$ is flat and $S>2$, then $\max\limits_{p\in M}S(p)>\frac{20}{9}$.
\end{quote}
A different but closely related development is due to Li--Zhao \cite{LiZhao26}, who constructed closed embedded flat minimal tori showing that Lu's refined quantity $S+a_2$ itself does not satisfy a Chern-type discreteness or second-gap statement in codimension at least three. Here ``flat'' refers to the induced surface metric in that construction, and their result concerns Lu's quantity rather than the constant-$S$ problem studied here. This construction explains why we focus on $S$ under the flat-normal-bundle assumption rather than on the refined quantity $S+a_2$. Very recently, as a byproduct, we \cite{GTZ} established the following first-gap theorem for $S$ under the flat-normal-bundle condition.
 
% \vspace{2mm}
  \begin{thm}[\cite{GTZ}]
 	Let $M^n$ be a closed minimal submanifold immersed in the unit sphere $\mathbb{S}^{n+m}$ with flat normal bundle. If $0\leqslant S\leqslant n$, then $M^n$ must be one of the following:
 	\begin{enumerate}
 		\item[\rm (1)]   the totally geodesic sphere, with $S\equiv 0$; 
 		\item[\rm (2)]   the Clifford torus $\mathbb{S}^{k}\left(\sqrt{\frac{k}{n}}\right)\times \mathbb{S}^{n-k}\left(\sqrt{\frac{n-k}{n}}\right)$, with $S\equiv n$ and $1\leqslant k\leqslant n-1$, lying, up to an ambient isometry, in a totally geodesic $\mathbb{S}^{n+1}\subset\mathbb{S}^{n+m}$.
 	\end{enumerate}
 	  \end{thm}
% \vspace{2mm}
 
As mentioned above, Ding--Ge--Li--Yang \cite{DingGeLiYang26} have already proved the second-gap theorem for $n=2$ under the flat-normal-bundle and constant-$S$ assumptions, and they further obtained the explicit second-gap constant $\frac{n}{9}$. However, in dimensions $n\geqslant3$ and codimensions $m\geqslant2$, no comparable explicit second-gap theorem for the constant-$S$ Chern-type problem was known, even under structural assumptions such as flatness of the normal bundle.
 In this paper, we establish an explicit second-gap theorem for minimal submanifolds with flat normal bundle in arbitrary codimension. To the best of our knowledge, this is the first explicit second-gap result for the constant-$S$ Chern-type problem in a genuinely higher-codimensional setting with $n\geqslant3$. 
 Set
 	\[
 	\delta(n,m)=\begin{cases}
 		\dfrac{n}{81},&m=2,\ 3\leqslant n\leqslant5,\\[2mm]
 		\dfrac{n}{62},&m=2,\ n\geqslant6,\\[2mm]
 		\dfrac{n}{87},&m\geqslant3,\ 3\leqslant n\leqslant5,\\[2mm]
 		\dfrac{n}{67},&m\geqslant3,\ n\geqslant6.
 	\end{cases}
 	\]
\begin{thm}\label{main}
	Let $M^n$ $(n\geqslant3)$ be a closed minimal submanifold in the unit sphere $\mathbb{S}^{n+m}$ $(m\geqslant2)$ with flat normal bundle. 
	If $S$ is constant and $0\leqslant S\leqslant n+\delta(n,m)$, 
	then $M^n$ must be one of the following:
	 	\begin{enumerate}
	 		\item[\rm (1)]   the totally geodesic sphere, with $S\equiv 0$; 
	 		\item[\rm (2)]   the Clifford torus $\mathbb{S}^{k}\left(\sqrt{\frac{k}{n}}\right)\times \mathbb{S}^{n-k}\left(\sqrt{\frac{n-k}{n}}\right)$, with $S\equiv n$ and $1\leqslant k\leqslant n-1$, lying, up to an ambient isometry, in a totally geodesic $\mathbb{S}^{n+1}\subset\mathbb{S}^{n+m}$.
	 	\end{enumerate}
%	then $S\equiv n$ and $M^{n}$ is a Clifford torus. %In particular, the constants in the cases $m\geqslant3$ are independent of the codimension, while the codimension-two case admits the sharper constants displayed above.
\end{thm}

\begin{rem}\label{rem:LiZhao-flat-tori}
The two-dimensional theorem quoted above has the second-gap constant $\frac n9$. Hence the conclusion of Theorem~\ref{main} remains valid in dimension two. % if $\delta$ is replaced by any of the smaller bounds $n/81$, $n/62$, $n/87$, or $n/67$. 
The flat-normal-bundle assumption is essential in dimension two. Without it, Li--Zhao~\cite{LiZhao26} constructed a sequence of closed embedded flat minimal tori in spheres, none of which is a Clifford torus, such that $S\equiv n$ and the normal scalar curvature is constant on each torus, with these constants converging to zero. This shows that small normal scalar curvature alone should not be expected to replace flatness of the normal bundle in Clifford-type rigidity results.
\end{rem}
This paper is organized as follows. In Section \ref{2}, by computing $\Delta S$ and $\Delta(|\nabla h|^2)$, we derive the Simons formula \eqref{Simons}, Simons-type formula \eqref{sum}, and introduce the Peng--Terng-type invariant $\sum\limits_{\alpha,\beta}(A_{\alpha,\beta}-2B_{\alpha,\beta})$ defined in \eqref{P-T type}. Moreover, we prove a sequence of sharp algebraic inequalities in Lemma \ref{inesharp}. In Section \ref{3}, we first establish an important property of flat normal bundle in Lemma \ref{fine}. Then combining the methods of Peng--Terng \cite{Peng-T2} and Lei--Xu--Xu \cite{2017}, Lemma \ref{mixA-2B} provides a new upper bound for the Peng--Terng-type invariant in higher codimension. Furthermore, we generalize an integral formula in Lemma \ref{AAA}. In Section \ref{4}, we mainly estimate both the upper and lower bounds for $\int_{M}|\nabla^2 h|^2dM$ in codimension two. These bounds, stated in Theorem \ref{upperco2} and Theorem \ref{thmlower1}, are sharper than the corresponding estimates in general codimension and lead to the proof of Theorem \ref{main} in codimension two. In Section \ref{5}, using a similar method, we complete the proof of Theorem \ref{main} in general codimension.

\section{Preliminaries}\label{2}
We consider an $n$-dimensional closed minimal submanifold $M^n$ immersed in the unit sphere $\mathbb{S}^{n+m}$. Let $\{e_1,\ldots,e_{n+m}\}$ be a local orthonormal frame on $T(\mathbb{S}^{n+m})$ such that, when restricted to $M^n$, the vectors $\{e_1,\ldots,e_n\}$  lie in the tangent bundle $T(M)$ and $\{e_{n+1},\ldots,e_{n+m}\}$ lie in the normal bundle $T^{\perp}(M)$, respectively. Let $h$ denote the second fundamental form of $M$. We shall make use of the following convention on the ranges of the indices:
\begin{equation*}
1\leqslant i,j,k,\ldots\leqslant n,
\qquad
1\leqslant r,s,t\leqslant m,
\qquad n+1\leqslant\alpha,\beta,\gamma\leqslant n+m.
\end{equation*}
For convenience, denote
\[
\begin{gathered}
|\nabla h^{\alpha}|^2=\sum_{i,j,k}(h_{ijk}^{\alpha})^2,
\qquad |\nabla h|^2=\sum\limits_{\alpha,i,j,k}(h_{ijk}^{\alpha})^2,\qquad
|\nabla^2 h|^2=\sum\limits_{\alpha,i,j,k,l}(h_{ijkl}^{\alpha})^2.
\end{gathered}
\]
Define
\[S:=|h|^2=\sum_{i,j,\alpha}(h_{ij}^{\alpha})^2,\quad \Delta h_{ij}^{\alpha}:=\sum_{k}h_{ijkk}^{\alpha}.\]
Denote by $A^{\alpha}$ the shape operator of $M^n$ with respect to a given normal orthonormal frame. For later use, we denote the fundamental matrix $\mathcal{A}=(\langle A^{\alpha},A^{\beta}\rangle)_{m\times m}$ and assume that $a_{1}\geqslant a_{2}\geqslant\cdots\geqslant a_m\geqslant0$ are eigenvalues of $\mathcal{A}$.
We further set
\begin{equation}\label{eq:bD-def}
	b:=\sum_{r=2}^m a_r=S-a_1,
	\qquad \sigma_2:=\sum_{r<s}a_r a_s,\qquad\sigma_3:=\sum_{r<s<t}a_r a_s a_t.
\end{equation}
Using the Einstein summation convention, the Ricci formulas are
\begin{equation*}
		h_{ijkl}^{\alpha}-h_{ijlk}^{\alpha}=h_{ip}^{\alpha}R_{pjkl}+h_{pj}^{\alpha}R_{pikl}-h_{ij}^{\beta}R_{\alpha\beta kl},
\end{equation*}
\begin{equation*}
		h_{ijklm}^{\alpha}-h_{ijkml}^{\alpha}=h_{pjk}^{\alpha}R_{pilm}+h_{ipk}^{\alpha}R_{pjlm}+h_{ijp}^{\alpha}R_{pklm}-h_{ijk}^{\beta}R_{\alpha\beta lm},
\end{equation*}
where
\begin{equation*}
	\begin{aligned}
		R_{ijkl}=&\delta_{ik}\delta_{jl}-\delta_{il}\delta_{jk}+h_{ik}^{\alpha}h_{jl}^{\alpha}-h_{il}^{\alpha}h_{jk}^{\alpha},\\
		R_{\alpha\beta kl}=&h_{ik}^{\alpha}h_{il}^{\beta}-h_{il}^{\alpha}h_{ik}^{\beta}.
	\end{aligned}
\end{equation*}
A direct computation gives the following formulas.
\begin{prop}
	Let $M^n$ be a closed minimal submanifold in the unit sphere $\mathbb{S}^{n+m}$. Then
\begin{equation}\label{Simons}
	\begin{aligned}
		\frac{1}{2}\Delta S=|\nabla h|^2+nS-\sum_{\alpha,\beta}|[A^{\alpha},A^{\beta}]|^2-\sum_{\alpha,\beta}|\langle A^{\alpha},A^{\beta}\rangle|^2.
	\end{aligned}
\end{equation}
\begin{equation}\label{sum}
	\begin{aligned}
		\frac{1}{2}\Delta(|\nabla h|^2)
		=&(2n+3)|\nabla h|^2+\sum_{i,j,k,p,l,\alpha,\beta}(6h_{ijk}^{\alpha}h_{lpk}^{\alpha}h_{pj}^{\beta}h_{il}^{\beta}\\
		&-3h_{ijk}^{\alpha}h_{ijp}^{\alpha}h_{pl}^{\beta}h_{lk}^{\beta}-6h_{ijk}^{\alpha}h_{pjl}^{\alpha}h_{pl}^{\beta}h_{ik}^{\beta}+6h_{ijk}^{\alpha}h_{lp}^{\alpha}h_{pjk}^{\beta}h_{il}^{\beta}\\
		&-6h_{ijk}^{\alpha}h_{pi}^{\alpha}h_{pl}^{\beta}h_{jlk}^{\beta}-h_{ijk}^{\alpha}h_{pl}^{\alpha}h_{pl}^{\beta}h_{ijk}^{\beta})+|\nabla^2 h|^2.
	\end{aligned}
\end{equation}
\end{prop}
\begin{proof}
	By the Ricci formula and the minimality of $M$, we have
	\begin{equation}\label{Lu}
		\Delta h_{ij}^{\alpha}=nh_{ij}^{\alpha}+2h_{kp}^{\alpha}h_{pj}^{\beta}h_{ik}^{\beta}-h_{kp}^{\alpha}h_{pk}^{\beta}h_{ij}^{\beta}-h_{pi}^{\alpha}h_{pk}^{\beta}h_{jk}^{\beta}-h_{ki}^{\beta}h_{pj}^{\alpha}h_{pk}^{\beta}.
	\end{equation}
In matrix notation, one has
\begin{equation}\label{matrix}
	\Delta A^{\alpha}=nA^{\alpha}-\langle A^{\alpha},A^{\beta}\rangle A^{\beta}-[A^{\beta},[A^{\beta},A^{\alpha}]].
\end{equation}
Hence, the Simons formula \eqref{Simons} follows directly from \eqref{matrix}.
By \eqref{Lu}, we obtain
\begin{equation*}
	\begin{aligned}
		(\Delta h_{ij}^{\alpha})_{k}=&nh_{ijk}^{\alpha}+2h_{lpk}^{\alpha}h_{pj}^{\beta}h_{il}^{\beta}+2h_{lp}^{\alpha}h_{pjk}^{\beta}h_{il}^{\beta}+2h_{lp}^{\alpha}h_{pj}^{\beta}h_{ilk}^{\beta}\\
		&-h_{lpk}^{\alpha}h_{pl}^{\beta}h_{ij}^{\beta}-h_{lp}^{\alpha}h_{plk}^{\beta}h_{ij}^{\beta}-h_{lp}^{\alpha}h_{pl}^{\beta}h_{ijk}^{\beta}\\
		&-h_{pik}^{\alpha}h_{pl}^{\beta}h_{jl}^{\beta}-h_{pi}^{\alpha}h_{plk}^{\beta}h_{jl}^{\beta}-h_{pi}^{\alpha}h_{pl}^{\beta}h_{jlk}^{\beta}\\
		&-h_{lik}^{\beta}h_{pj}^{\alpha}h_{pl}^{\beta}-h_{li}^{\beta}h_{pjk}^{\alpha}h_{pl}^{\beta}-h_{li}^{\beta}h_{pj}^{\alpha}h_{plk}^{\beta}. 
	\end{aligned}
\end{equation*}
Since $M$ is minimal, we get
\begin{equation*}
	\begin{aligned}
	h_{ipl}^{\alpha}R_{pjkl}=&h_{ijk}^{\alpha}+h_{ipl}^{\alpha}h_{pk}^{\beta}h_{jl}^{\beta}-h_{ipl}^{\alpha}h_{pl}^{\beta}h_{jk}^{\beta},\\
	h_{pjl}^{\alpha}R_{pikl}=&h_{ijk}^{\alpha}+h_{pjl}^{\alpha}h_{pk}^{\beta}h_{il}^{\beta}-h_{pjl}^{\alpha}h_{pl}^{\beta}h_{ik}^{\beta},\\
	h_{ijl}^{\beta}R_{\alpha\beta kl}=&h_{ijl}^{\beta}h_{pk}^{\alpha}h_{pl}^{\beta}-h_{ijl}^{\beta}h_{pl}^{\alpha}h_{pk}^{\beta},\\
		h_{ijp}^{\alpha}R_{plkl}=&(n-1)h_{ijk}^{\alpha}+h_{ijp}^{\alpha}h_{pk}^{\beta}h_{ll}^{\beta}-h_{ijp}^{\alpha}h_{pl}^{\beta}h_{kl}^{\beta}\\
		=&(n-1)h_{ijk}^{\alpha}-h_{ijp}^{\alpha}h_{pl}^{\beta}h_{kl}^{\beta},\\
		h_{ip}^{\alpha}R_{pjkll}=&h_{ip}^{\alpha}h_{pkl}^{\beta}h_{jl}^{\beta}+h_{ip}^{\alpha}h_{pk}^{\beta}h_{jll}^{\beta}-h_{ip}^{\alpha}h_{pll}^{\beta}h_{jk}^{\beta}-h_{ip}^{\alpha}h_{pl}^{\beta}h_{jkl}^{\beta}\\
		=&h_{ip}^{\alpha}h_{pkl}^{\beta}h_{jl}^{\beta}-h_{ip}^{\alpha}h_{pl}^{\beta}h_{jkl}^{\beta},\\
		h_{pj}^{\alpha}R_{pikll}=&h_{pj}^{\alpha}h_{pkl}^{\beta}h_{il}^{\beta}+h_{pj}^{\alpha}h_{pk}^{\beta}h_{ill}^{\beta}-h_{pj}^{\alpha}h_{pll}^{\beta}h_{ik}^{\beta}-h_{pj}^{\alpha}h_{pl}^{\beta}h_{ikl}^{\beta}\\
		=&h_{pj}^{\alpha}h_{pkl}^{\beta}h_{il}^{\beta}-h_{pj}^{\alpha}h_{pl}^{\beta}h_{ikl}^{\beta},\\
		h_{ij}^{\beta}R_{\alpha\beta kll}=&h_{ij}^{\beta}h_{pkl}^{\alpha}h_{pl}^{\beta}+h_{ij}^{\beta}h_{pk}^{\alpha}h_{pll}^{\beta}-h_{ij}^{\beta}h_{pll}^{\alpha}h_{pk}^{\beta}-h_{ij}^{\beta}h_{pl}^{\alpha}h_{pkl}^{\beta}\\
		=&h_{ij}^{\beta}h_{pkl}^{\alpha}h_{pl}^{\beta}-h_{ij}^{\beta}h_{pl}^{\alpha}h_{pkl}^{\beta}.
	\end{aligned}
\end{equation*}
Therefore, combining the equations above, we obtain
\begin{align}
		\Delta h_{ijk}^{\alpha}=\sum_{l}h_{ijkll}^{\alpha}=&\sum_{l}\left(h_{ijlk}^{\alpha}+h_{ip}^{\alpha}R_{pjkl}+h_{pj}^{\alpha}R_{pikl}-h_{ij}^{\beta}R_{\alpha\beta kl}\right)_l\label{ric}\\
		=&(\Delta h_{ij}^{\alpha})_{k}+2h_{ipl}^{\alpha}R_{pjkl}+2h_{pjl}^{\alpha}R_{pikl}-2h_{ijl}^{\beta}R_{\alpha\beta kl}+h_{ijp}^{\alpha}R_{plkl}\notag\\
		&+h_{ip}^{\alpha}R_{pjkll}+h_{pj}^{\alpha}R_{pikll}-h_{ij}^{\beta}R_{\alpha\beta kll}\notag\\
		=&(2n+3)h_{ijk}^{\alpha}+2\sum_{p,l}(h_{lpk}^{\alpha}h_{pj}^{\beta}h_{il}^{\beta}+h_{lp}^{\alpha}h_{pjk}^{\beta}h_{il}^{\beta}+h_{lp}^{\alpha}h_{pj}^{\beta}h_{ilk}^{\beta}\notag\\
		&+h_{ipl}^{\alpha}h_{pk}^{\beta}h_{jl}^{\beta}+h_{pjl}^{\alpha}h_{pk}^{\beta}h_{il}^{\beta}+h_{lp}^{\alpha}h_{pk}^{\beta}h_{ijl}^{\beta}-h_{pi}^{\alpha}h_{pl}^{\beta}h_{jlk}^{\beta}-h_{pjl}^{\alpha}h_{pl}^{\beta}h_{ik}^{\beta}\notag\\
		&-h_{ipl}^{\alpha}h_{pl}^{\beta}h_{jk}^{\beta}-h_{pj}^{\alpha}h_{pl}^{\beta}h_{ikl}^{\beta}-h_{pk}^{\alpha}h_{pl}^{\beta}h_{ijl}^{\beta}-h_{lpk}^{\alpha}h_{pl}^{\beta}h_{ij}^{\beta})\notag\\&
		-\sum_{p,l}(h_{lp}^{\alpha}h_{pl}^{\beta}h_{ijk}^{\beta}
		+h_{pik}^{\alpha}h_{pl}^{\beta}h_{jl}^{\beta}+h_{pjk}^{\alpha}h_{li}^{\beta}h_{pl}^{\beta}+h_{ijp}^{\alpha}h_{pl}^{\beta}h_{lk}^{\beta}).\notag
\end{align}
Since
\[\frac{1}{2}\Delta(|\nabla h|^2)=\sum_{i,j,k,\alpha}h_{ijk}^{\alpha}\Delta h_{ijk}^{\alpha}+|\nabla^2 h|^2,\]
using \eqref{ric}, a direct computation yields \eqref{sum}.
\end{proof}
From now on, we consider the case that the normal bundle of $M$ is flat.
\begin{prop}
	Let $M^n$ be a closed minimal submanifold in $\mathbb{S}^{n+m}$ with flat normal bundle. If $S$ is constant and $n\leqslant S\leqslant n+\delta\leqslant2n$, then
	\begin{equation}\label{Sconstant-general}
		|\nabla h|^2=-nS+|\mathcal{A}|^2=S(S-n)-2\sigma_2,
	\end{equation}
	and
	\begin{equation}\label{eq:b-eps}
		0\leqslant b\leqslant \frac{S-\sqrt{S(2n-S)}}{2}\leqslant\frac{n+\delta-\sqrt{n^2-\delta^2}}{2}=:\epsilon.
	\end{equation}
\end{prop}

\begin{proof}
	Since the normal bundle of $M$ is flat, we can choose a local orthonormal frame $\{e_{1},e_{2},\dots,e_{n+m}\}$ so that at any fixed point 
		$\langle A^\alpha, A^\beta\rangle=0$ for all $\alpha\neq\beta$ and
		$h_{ij}^\alpha=\lambda_i^\alpha\delta_{ij}$ for every $\alpha$. Thus
		\[|\mathcal{A}|^2=S^2-2\sigma_2.\]
	Then by \eqref{Simons}, since $S$ is constant, \eqref{Sconstant-general} follows. Hence, $|\nabla h|^2\geqslant 0$ yields $|\mathcal{A}|^2\geqslant nS$. On the other hand,
	$|\mathcal{A}|^2=\sum\limits_{r=1}^{m} a_r^2\leqslant a_1\sum\limits_{r=1}^{m} a_r=a_1S$, hence $a_1\geqslant n$. Since $S\leqslant 2n$, we have 
	\begin{equation}\label{b}
		0\leqslant b=S-a_1\leqslant S-n\leqslant \frac{S}{2}.
	\end{equation} 
 Moreover,
	\[
	nS\leqslant|\mathcal{A}|^2=\sum_{r=1}^{m} a_r^2
	=a_1^2+\sum_{r\geqslant2}a_r^2
	\leqslant(S-b)^2+b^2.
	\]
	Thus $2b^2-2Sb+S(S-n)\geqslant0$. By \eqref{b}, this quadratic inequality yields 
	\[b\leqslant \frac{S-\sqrt{S(2n-S)}}{2}.\] The function
	$g(S)=\frac{S-\sqrt{S(2n-S)}}{2}$ is increasing on $[n,n+\delta]$, so
	\eqref{eq:b-eps} holds.
\end{proof}
Following the Peng--Terng invariant in the hypersurface case, we define
\begin{equation}\label{P-T type}
A_{\alpha,\beta}:=\sum_{i,j,k,l,p}h_{ijk}^{\beta}h_{ijp}^{\beta}h_{pl}^{\alpha}h_{lk}^{\alpha}, \quad B_{\alpha,\beta}:=\sum_{i,j,k,l,p}h_{ijk}^{\beta}h_{lpk}^{\beta}h_{pj}^{\alpha}h_{il}^{\alpha}.
\end{equation}
Thus, $\sum\limits_{\alpha,\beta}A_{\alpha,\beta}$ and $\sum\limits_{\alpha,\beta}B_{\alpha,\beta}$ are globally smooth functions, independent of the choice of local frames.
\begin{prop}
	Let $M^n$ be a closed minimal submanifold in the unit sphere $\mathbb{S}^{n+m}$ with flat normal bundle. Then
\begin{equation}\label{eq:simons-gradient}
	\begin{aligned}
		\frac{1}{2}\Delta(|\nabla h|^2)=&(2n+3)|\nabla h|^2-3\sum_{\alpha,\beta}(A_{\alpha,\beta}-2B_{\alpha,\beta})-6\sum_{k,\alpha,\beta}\langle A^{\alpha},\nabla_{e_k}A^{\beta}\rangle^2\\
		&-\sum_{\alpha,\beta}\langle A^{\alpha},A^{\beta}\rangle\sum_{k}\langle\nabla_{e_k}A^{\alpha},\nabla_{e_k}A^{\beta}\rangle+|\nabla^2 h|^2.
	\end{aligned}
\end{equation}
\end{prop}
\begin{proof}
	A direct expansion gives
	\begin{equation}\label{vanish}
		\begin{aligned}
		\sum_{i,j,k,p,l,\alpha,\beta}(6h_{ijk}^{\alpha}h_{lp}^{\alpha}h_{pjk}^{\beta}h_{il}^{\beta}-6h_{ijk}^{\alpha}h_{pi}^{\alpha}h_{pl}^{\beta}h_{jlk}^{\beta})=&6\sum_{\alpha,\beta}\langle[A^{\alpha},A^{\beta}],\sum_{k}\nabla_{e_k}A^{\beta}\nabla_{e_k}A^{\alpha}\rangle\\
		=&3\sum_{\alpha,\beta}\langle[A^{\alpha},A^{\beta}],\sum_{k}[\nabla_{e_k}A^{\beta},\nabla_{e_k}A^{\alpha}]\rangle.
		\end{aligned}
	\end{equation}
	Since the normal bundle of $M$ is flat, (\ref{vanish}) vanishes. Then (\ref{sum}) becomes
	\begin{equation}\label{high}
		\begin{aligned}
			\frac{1}{2}\Delta(|\nabla h|^2)
			=&(2n+3)|\nabla h|^2+\sum_{i,j,k,p,l,\alpha,\beta}(6h_{ijk}^{\alpha}h_{lpk}^{\alpha}h_{pj}^{\beta}h_{il}^{\beta}\\
			&-3h_{ijk}^{\alpha}h_{ijp}^{\alpha}h_{pl}^{\beta}h_{lk}^{\beta}-6h_{ijk}^{\alpha}h_{pjl}^{\alpha}h_{pl}^{\beta}h_{ik}^{\beta}-h_{ijk}^{\alpha}h_{pl}^{\alpha}h_{pl}^{\beta}h_{ijk}^{\beta})+|\nabla^2 h|^2.
		\end{aligned}
	\end{equation}
Thus, \eqref{eq:simons-gradient} follows directly from \eqref{high}.
\end{proof}
We also need the following sharp algebraic inequality for subsequent estimates.
\begin{lem}\label{inesharp}
	Let $n\geqslant3$, and suppose that $x=(x_i)$ and $y=(y_i)$ in $\mathbb R^n$ satisfy
	\[
	\sum_i x_i=\sum_i y_i=0,
	\qquad \sum_i x_i y_i=0,
	\qquad \sum_i x_i^2=X,
	\qquad \sum_i y_i^2=Y.
	\]
	Then
	\begin{align}
		\sum_i x_i^2y_i^2&\leqslant \frac{n-2}{2n}XY, \label{eq:opt-three}
	\end{align}
	where equality holds if and only if $x_{i}y_{j}+x_{j}y_{i}=-\frac{2}{n-2}(x_{i}y_{i}+x_{j}y_{j})$ for $i\neq j$.\\
In addition,
	\begin{align}
		 |x_i y_i|&\leqslant \frac{n-1}{2n}\sqrt{XY}, \label{eq:opt-same-index}\\
		 |x_i y_j|&\leqslant \frac{n-1+\sqrt{n(n-2)}}{2n}\sqrt{XY},\quad i\ne j. \label{eq:opt-two}
	\end{align}
\end{lem}
\begin{proof}
	Define $d_i:=x_i y_i$ and
	$s_{ij}:=x_i y_j+x_j y_i$ for $i\ne j$.  Then
	\begin{equation}\label{Ri}
	\sum_i d_i=0,\qquad R_{i}:=\sum_{j\ne i}s_{ij}=-2d_i,
	\end{equation}
	and
	\begin{equation}\label{eq:offdiag-identity}
		\begin{aligned}
			\sum_{i<j}s_{ij}^2=&\sum_{i<j}(x_i^2y_j^2+x_j^2y_i^2+2x_{i}x_{j}y_{i}y_{j})\\
			=&\sum_{i\neq j}x_i^2y_j^2+(\sum_{i}x_{i}y_{i})^2-\sum_{i}x_{i}^2y_{i}^2\\
			=&XY-2\sum_i x_i^2y_i^2.
		\end{aligned}
	\end{equation}
	It follows from \eqref{Ri}, and by expanding
	$\sum\limits_{i<j}(s_{ij}-\frac{R_{i}+R_{j}}{n-2})^2\geqslant0$, the sharp Cauchy--Schwarz inequality on the complete graph says that
	\[
	\sum_{i<j}s_{ij}^2\geqslant\frac1{n-2}\sum_iR_i^2.
	\]
	Applying this with $R_i=-2d_i$ and
	using \eqref{eq:offdiag-identity} gives
	\[
	XY-2\sum_i x_i^2y_i^2\geqslant\frac{4}{n-2}\sum_i x_i^2y_i^2,
	\]
	so \eqref{eq:opt-three} holds. Moreover, equality holds if and only if $s_{ij}=\frac{R_{i}+R_{j}}{n-2}$ for $i\neq j$.
	
	For \eqref{eq:opt-same-index} and \eqref{eq:opt-two}, the case $XY=0$ is trivial. We therefore assume $X,Y>0$ and normalize. Set
			\[
			\tilde x_i := \frac{x_i}{\sqrt X}, \qquad 
			\tilde y_i := \frac{y_i}{\sqrt Y}.
			\]
			Then
			\[
			\sum_i \tilde x_i = \sum_i \tilde y_i = 0,\qquad 
			\sum_i \tilde x_i^2 = \sum_i \tilde y_i^2 = 1,\qquad 
			\sum_i \tilde x_i \tilde y_i = 0.
			\]
			It suffices to prove
			\[
			|\tilde x_i \tilde y_i| \leqslant \frac{n-1}{2n},
			\]
			because multiplying by $\sqrt{XY}$ gives the general case.
			Let 
			\[
			H := \left\{ z = (z_1,\dots,z_n) \in \mathbb R^n : \sum_{j=1}^n z_j = 0 \right\},
			\]
			which is an $(n-1)$-dimensional subspace of $\mathbb R^n$. For each index $i$, define
			\[
			v_i := e_i - \frac{1}{n}(1,1,\dots,1),
			\]
			where $e_i$ is the $i$-th standard basis vector. Clearly $v_i \in H$, and
			\[
			\|v_i\|^2 = \frac{n-1}{n^2}+\frac{(n-1)^2}{n^2} = \frac{n-1}{n}. 
			\]
			Since both $\tilde x = (\tilde x_i)$ and $\tilde y = (\tilde y_i)$ lie in $H$, for any $z \in H$ we have
			\[
			z_i = \langle z, e_i \rangle = \langle z, v_i \rangle,
			\]
			because $\langle z, \frac{1}{n}(1,\dots,1)\rangle = \frac{1}{n}\sum\limits_j z_j = 0$. 			
			Moreover, $\sum\limits_i \tilde x_i \tilde y_i = 0$ means $\langle \tilde x, \tilde y \rangle = 0$, so $\tilde x$ and $\tilde y$ are orthonormal vectors in $H$. 
			Notice that
			\[
			\tilde x_i = \langle \tilde x, v_i \rangle, \qquad 
			\tilde y_i = \langle \tilde y, v_i \rangle. 
			\]
			Therefore,
			\[
			\tilde x_i^2 + \tilde y_i^2 \leqslant \|v_i\|^2 = \frac{n-1}{n}.
			\]
			Thus,
			\[
			|\tilde x_i \tilde y_i| \leqslant \frac{1}{2}(\tilde x_i^2 + \tilde y_i^2)\leqslant\frac{n-1}{2n},
			\]
  which	proves \eqref{eq:opt-same-index}.

		For \eqref{eq:opt-two}, it is enough to prove
					\[
					|\tilde x_i \tilde y_j| \leqslant \frac{n-1+\sqrt{n(n-2)}}{2n} \qquad (i\neq j).
					\]
		For $i\neq j$, the vectors $v_i$ and $v_j$ have the same norm:
					\[
					\|v_i\|^2 = \|v_j\|^2 = \frac{n-1}{n},
					\]
		and their inner product is
					\[
					\langle v_i, v_j \rangle = -\frac{1}{n}.
					\]
		Therefore the angle $\phi$ between $v_i$ and $v_j$ satisfies
					\[
					\cos\phi = \frac{\langle v_i, v_j \rangle}{\|v_i\|\|v_j\|}
					= -\frac{1}{n-1},
					\]
			so
					\[
					|\sin\phi| = \sqrt{1 - \frac{1}{(n-1)^2}}
					= \frac{\sqrt{n(n-2)}}{n-1}.
					\]
					
We shall use the following elementary optimization. For two fixed vectors
$u,v$ in an inner product space and for any orthonormal pair $x,y$, one has
\[
|\langle x,u\rangle\langle y,v\rangle|
\leqslant \frac{\|u\|\|v\|}{2}\bigl(1+|\sin\phi|\bigr),
\]
where $\phi$ is the angle between $u$ and $v$. Indeed, set
$a=\langle x,u\rangle$, $b=\langle y,v\rangle$,
$c=\langle y,u\rangle$, and $d=\langle x,v\rangle$. Then
\[
2|ab|\leqslant |ab+cd|+|ab-cd|.
\]
The first term satisfies
\[
|ab+cd|
=\big|\langle u,\langle y,v\rangle x+\langle x,v\rangle y\rangle\big|
\leqslant \|u\|\sqrt{\langle y,v\rangle^2+\langle x,v\rangle^2}
\leqslant \|u\|\|v\|.
\]
For the second term, using the exterior product gives
\[
|ab-cd|=|\langle x\wedge y,u\wedge v\rangle|
\leqslant \|u\wedge v\|
=\|u\|\|v\|\,|\sin\phi|.
\]
Combining these two estimates proves the claimed inequality.

Applying this with $u=v_i$ and $v=v_j$, we obtain
					\[
					|\tilde x_i \tilde y_j|
					= |\langle \tilde x, v_i\rangle \langle \tilde y, v_j\rangle|
					\leqslant \frac{\|v_i\|\|v_j\|}{2}\left(1 + \frac{\sqrt{n(n-2)}}{n-1}\right).
					\]
					Since $\|v_i\|\|v_j\| = \frac{n-1}{n}$, this becomes
					\[
					|\tilde x_i \tilde y_j|
					\leqslant \frac{n-1}{2n}\left(1 + \frac{\sqrt{n(n-2)}}{n-1}\right)
					= \frac{n-1+\sqrt{n(n-2)}}{2n},
					\]
				which proves \eqref{eq:opt-two}.
\end{proof}
\section{Some lemmas and key estimates}\label{3}
Near any given point $q\in M$, since the normal bundle of $M$ is flat, we can choose a local orthonormal frame $\{e_{1},e_{2},\dots,e_{n+m}\}$ so that at $q$ 
\begin{equation}\label{eq:diag}
	\langle A^\alpha, A^\beta\rangle=0\quad \text{for }\alpha\neq\beta,\qquad
	h_{ij}^\alpha=\lambda_i^\alpha\delta_{ij}\quad\text{for all }\alpha.
\end{equation}
We first record a useful consequence of the flatness of the normal bundle.
\begin{lem}\label{fine}
	Let $M^n$ be a closed minimal submanifold in the unit sphere $\mathbb{S}^{n+m}$ with flat normal bundle. Then for any $k$, $\alpha$ and $\beta$ we have
	\begin{equation}\label{flatness-plus}
		\langle A^{\alpha},\nabla_{e_k}A^{\beta}\rangle=\sum_{i}\lambda_{i}^{\alpha}h_{iik}^{\beta}=\sum_{i}\lambda_{i}^{\beta}h_{iik}^{\alpha}=\langle A^{\beta},\nabla_{e_k}A^{\alpha}\rangle.
	\end{equation}
\end{lem}
\begin{proof}
	Since the normal bundle of $M$ is flat,
	\[R_{\alpha\beta ki}=\sum_{j}(h_{kj}^{\alpha}h_{ji}^{\beta}-h_{kj}^{\beta}h_{ji}^{\alpha})=0.\]
	Covariantly differentiating this tensorial identity at the chosen point, for every $l$, by \eqref{eq:diag}, we obtain
	\begin{align*}
		0=R_{\alpha\beta kil}=&\sum_{j}(h_{kjl}^{\alpha}h_{ji}^{\beta}+h_{kj}^{\alpha}h_{jil}^{\beta}-h_{kjl}^{\beta}h_{ji}^{\alpha}-h_{kj}^{\beta}h_{jil}^{\alpha})\\
		=&\lambda_{i}^{\beta}h_{kil}^{\alpha}+\lambda_{k}^{\alpha}h_{kil}^{\beta}-\lambda_{i}^{\alpha}h_{kil}^{\beta}-\lambda_{k}^{\beta}h_{kil}^{\alpha}\\
		=&(\lambda_{i}^{\beta}-\lambda_{k}^{\beta})h_{kil}^{\alpha}+(\lambda_{k}^{\alpha}-\lambda_{i}^{\alpha})h_{kil}^{\beta}.
	\end{align*}
	Hence, 
	\begin{equation}\label{letandsum}
		(\lambda_{i}^{\beta}-\lambda_{k}^{\beta})h_{kil}^{\alpha}=(\lambda_{i}^{\alpha}-\lambda_{k}^{\alpha})h_{kil}^{\beta}.
	\end{equation}
	Letting $l=i$ in \eqref{letandsum} and summing over $i$ yields
	\[\sum_{i}(\lambda_{i}^{\beta}-\lambda_{k}^{\beta})h_{iik}^{\alpha}=\sum_{i}(\lambda_{i}^{\alpha}-\lambda_{k}^{\alpha})h_{iik}^{\beta}.\]
	Since the minimality of $M$ implies $\sum\limits_{i}h_{iik}^{\alpha}=0$ for any $k$ and $\alpha$,
	\eqref{flatness-plus} follows.
\end{proof}
For further estimates, we give the following lemma (cf.\ Lemma 2.7 in \cite{GL}).
\begin{lem}
 Using the notation above, for any $i$ and $\alpha$, we have
 \begin{equation}\label{use}
 	(\lambda_{i}^{\alpha})^2\leqslant\frac{n-1}{n}\|A^{\alpha}\|^2.
 \end{equation}
\end{lem}
\begin{proof}
Since $M$ is minimal, for any fixed $i$ and $\alpha$, we have 
\begin{equation*}
	\begin{aligned}
	0=\sum_{j,k}\lambda_{j}^{\alpha}\lambda_{k}^{\alpha}=&(\lambda_{i}^{\alpha})^2+2\lambda_{i}^{\alpha}\sum_{k\neq i}\lambda_{k}^{\alpha}+\sum_{j,k\neq i}\lambda_{j}^{\alpha}\lambda_{k}^{\alpha}\\
	\leqslant&-(\lambda_{i}^{\alpha})^2+\sum_{j,k\neq i}\frac{(\lambda_{j}^{\alpha})^2+(\lambda_{k}^{\alpha})^2}{2}\\
	=&(n-1)\|A^{\alpha}\|^2-n(\lambda_{i}^{\alpha})^2.
	\end{aligned}
\end{equation*}
Hence, \eqref{use} holds.
\end{proof}
In 1983, Peng--Terng \cite{Peng-T2} gave a pointwise estimate for $A-2B$ as follows.
\begin{lem}[\cite{Peng-T2}]\label{1983}
	Let $M$ be an $n$-dimensional closed minimal hypersurface in $\mathbb{S}^{n+1}$. Then
	\begin{equation*}
		3(A-2B)\leqslant \frac{\sqrt{17}+1}{2}S|\nabla h|^2,
	\end{equation*}
where $A=\sum\limits_{i,j,k}\lambda_{i}^2h_{ijk}^2$, $B=\sum\limits_{i,j,k}\lambda_{i}\lambda_{j}h_{ijk}^2$.
\end{lem}
Set
\begin{equation}\label{eq:Gamma-n}
	\Gamma_n:=\frac{n+3+\sqrt{17n^2-26n+9}}{2n}.
\end{equation}
Using the same method as in the proof of Lemma~\ref{1983}, we obtain the following result.
\begin{lem}
	Let $M^{n}$ $(n\geqslant3)$ be a closed minimal submanifold in $\mathbb{S}^{n+m}$. Then for any fixed $\alpha$ and $\beta$,
	\begin{equation}\label{key}
			3\left(A_{\alpha,\beta}-2B_{\alpha,\beta}\right)\leqslant \Gamma_n\|A^{\alpha}\|^2|\nabla h^{\beta}|^2,
	\end{equation}
where $A_{\alpha,\beta}$ and $B_{\alpha,\beta}$ are defined in \eqref{P-T type}.
\end{lem}
\begin{proof}
	For any fixed $\alpha$ and $\beta$, we have a pointwise estimate
	\begin{equation}\label{peng}
	\begin{aligned}
		3(A_{\alpha,\beta}-2B_{\alpha,\beta})
		=&\sum_{i,j,k} (h_{ijk}^{\beta})^2 [(\lambda_i^{\alpha})^2+(\lambda_j^{\alpha})^2+(\lambda_k^{\alpha})^2 - 2\lambda_i^{\alpha}\lambda_j^{\alpha} - 2\lambda_j^{\alpha}\lambda_k^{\alpha} - 2\lambda_i^{\alpha}\lambda_k^{\alpha}] \\
		=&-\sum_i 3(h_{iii}^{\beta})^2(\lambda_i^{\alpha})^2 + \sum_{i\neq j} 3(h_{iij}^{\beta})^2 \big((\lambda_j^{\alpha})^2-4\lambda_j^{\alpha}\lambda_i^{\alpha}\big) \\
		&+\sum_{i, j, k \;\mathrm{distinct}} (h_{ijk}^{\beta})^2 \big(2(\lambda_i^{\alpha})^2+2(\lambda_j^{\alpha})^2+2(\lambda_k^{\alpha})^2-(\lambda_{i}^{\alpha}+\lambda_{j}^{\alpha}+\lambda_{k}^{\alpha})^2\big) \\
		\leqslant&2\|A^{\alpha}\|^2\sum_{i, j, k \;\mathrm{distinct}} (h_{ijk}^{\beta})^2+\sum_{i\neq j} 3(h_{iij}^{\beta})^2 \big((\lambda_j^{\alpha})^2-4\lambda_j^{\alpha}\lambda_i^{\alpha}\big),
	\end{aligned}
\end{equation}
Because $\Gamma_n\geqslant2$, it remains to prove that for $i\neq j$,
\[(\lambda_j^{\alpha})^2-4\lambda_j^{\alpha}\lambda_i^{\alpha}\leqslant \Gamma_n\|A^{\alpha}\|^2.\]
For $i\neq j$, since $\sum\limits_l\lambda_l^{\alpha}=0$, the Cauchy inequality gives
\[
(n-2)\sum_{l\ne i,j}(\lambda_l^{\alpha})^2
\geqslant \big(\sum_{l\ne i,j}\lambda_l^{\alpha}\big)^2=(\lambda_i^{\alpha}+\lambda_j^{\alpha})^2.
\]
Therefore
\[
\sum_l(\lambda_l^{\alpha})^2
\geqslant (\lambda_i^{\alpha})^2+(\lambda_j^{\alpha})^2+\frac{1}{n-2}(\lambda_i^{\alpha}+\lambda_j^{\alpha})^2.
\]
We now maximize the quadratic form
\[
(\lambda_j^{\alpha})^2-4\lambda_j^{\alpha}\lambda_i^{\alpha}
\]
under the normalization
\[
(\lambda_i^{\alpha})^2+(\lambda_j^{\alpha})^2+\frac{1}{n-2}(\lambda_i^{\alpha}+\lambda_j^{\alpha})^2=1.
\]
Equivalently, this is the largest generalized eigenvalue of the pair
\[
P=\begin{pmatrix}0&-2\\-2&1\end{pmatrix},\qquad
Q=\begin{pmatrix}\frac{n-1}{n-2}&\frac1{n-2}\\[1mm]
	\frac1{n-2}&\frac{n-1}{n-2}\end{pmatrix}.
\]
The characteristic equation is
\[
\det(P-\gamma Q)=0,
\]
that is,
\[
n\gamma^2-(n+3)\gamma-4(n-2)=0.
\]
The positive root is precisely
\[
\gamma=\Gamma_n=\frac{n+3+\sqrt{17n^2-26n+9}}{2n}.
\]
Hence
\begin{equation}\label{polyGamma}
(\lambda_j^{\alpha})^2-4\lambda_j^{\alpha}\lambda_i^{\alpha}
\leqslant \Gamma_n\sum_l(\lambda_l^{\alpha})^2
=\Gamma_n\|A^{\alpha}\|^2.
\end{equation}
Combining \eqref{peng} with \eqref{polyGamma} gives \eqref{key}.
\end{proof}
\begin{rem}
	It follows immediately that \[\Gamma_n\leqslant \frac{\sqrt{17}+1}{2}, \quad\lim_{n \to \infty}\Gamma_n=\frac{\sqrt{17}+1}{2}.\]
\end{rem}
The following algebraic inequalities proved by Lei--Xu--Xu \cite{2017} will be useful in later estimates.
\begin{lem}[\cite{2017}]\label{lem:XLX-triple}
	For all $x,y,z\in\mathbb R$,
	\begin{align*}
		-2(xy+yz+zx+2)^3
		<&(x-y)^2(xy+1)^2+(x-z)^2(xz+1)^2+(z-y)^2(yz+1)^2.\nonumber
	\end{align*}
\end{lem}
\begin{lem}[\cite{2017}]\label{lem:XLX-pair}
	Let $s\geqslant6$ and $x^2+y^2\leqslant s$.  Then
	\begin{equation*}
		-(x^2+4xy+4)^3<\frac{16}{5}\left(3-\frac{10}{s}\right)(x-y)^2(1+xy)^2.
	\end{equation*}
\end{lem}
Define
\begin{equation}\label{F}
	F_{1}:=\sum\limits_{i, j}(\lambda_{i}^{n+1}-\lambda_{j}^{n+1})^2(1+\lambda_{i}^{n+1}\lambda_{j}^{n+1})^2,\quad F:=\sum\limits_{i, j,\alpha}(\lambda_{i}^{\alpha}-\lambda_{j}^{\alpha})^2(1+\sum\limits_{\beta}\lambda_{i}^{\beta}\lambda_{j}^{\beta})^2.
\end{equation}
We obtain a new upper bound for $\sum\limits_{\alpha,\beta}(A_{\alpha,\beta}-2B_{\alpha,\beta})$ in the following lemma.
\begin{lem}\label{mixA-2B}
	Let $M^{n}$ $(n\geqslant3)$ be a closed minimal submanifold in $\mathbb{S}^{n+m}$ with flat normal bundle. Suppose that $S$ is constant and
$
		n\leqslant S\leqslant n+\delta\leqslant2n
$
	for some $\delta>0$. Then, for every $\rho>0$,
\begin{equation}\label{sumA-2B}
	\begin{aligned}
	&3\sum_{\alpha,\beta}(A_{\alpha,\beta}-2B_{\alpha,\beta})\\
	\leqslant&\left(S+4+(\Gamma_n-1)\epsilon+\sqrt[3]{c(1+\rho)F+c(1+\rho^{-1})\frac{n-2}{n}(S-\epsilon)\epsilon^2}\right)|\nabla h|^2,
\end{aligned}
\end{equation}
where
$
\epsilon:=\frac{n+\delta-\sqrt{n^2-\delta^2}}{2}.
$
Moreover, 
\[
c=\begin{cases}
	\frac{32}{15},&3\leqslant n\leqslant5,\\[1mm]
	\frac{24}{5}-\frac{16}{n+\delta},&n\geqslant6.
\end{cases}\]
\end{lem}
\begin{proof}
	When $n\geqslant 6$, since $S$ is constant, combining
	\[0\leqslant|\nabla h|^2=-nS+|\mathcal{A}|^2\] and
	\[|\mathcal{A}|^2=\sum_{r=1}^{m}a_{r}^2\leqslant a_{1}\sum_{r=1}^{m}a_{r}=a_{1}S,\]
	we obtain $a_{1}\geqslant n$ as $S>0$. On the other hand, $a_{1}\leqslant S\leqslant n+\delta$.
	
	In the frame chosen in \eqref{eq:diag}, $a_{r}=\|A^{n+r}\|^2$, $r=1,2,\ldots,m$.
	From the definition of $A_{\alpha,\beta}$ and $B_{\alpha,\beta}$, we have
\begin{equation}\label{decomp}
	\begin{aligned}
		{}&\ 3(A_{n+1,\beta}-2B_{n+1,\beta})\\
		={}&-\sum_i 3(h_{iii}^{\beta})^2(\lambda_i^{n+1})^2
		+\sum_{i\neq j}3(h_{iij}^{\beta})^2
		\big[(\lambda_j^{n+1})^2-4\lambda_j^{n+1}\lambda_i^{n+1}\big]\\
		&+\sum_{i, j, k \;\mathrm{distinct}}(h_{ijk}^{\beta})^2
		\big[(\lambda_i^{n+1})^2+(\lambda_j^{n+1})^2+(\lambda_k^{n+1})^2\\
		&\qquad\qquad -2\lambda_i^{n+1}\lambda_j^{n+1}
		-2\lambda_j^{n+1}\lambda_k^{n+1}
		-2\lambda_i^{n+1}\lambda_k^{n+1}\big].
	\end{aligned}
\end{equation}
	For distinct $i,j,k$, Lemma~\ref{lem:XLX-triple} implies
	\begin{align*}
	&-2\lambda_i^{n+1}\lambda_j^{n+1}-2\lambda_j^{n+1}\lambda_k^{n+1}-2\lambda_i^{n+1}\lambda_k^{n+1}-4\\
	<&[4(\lambda_{i}^{n+1}-\lambda_{j}^{n+1})^2(1+\lambda_{i}^{n+1}\lambda_{j}^{n+1})^2+4(\lambda_{j}^{n+1}-\lambda_{k}^{n+1})^2(1+\lambda_{j}^{n+1}\lambda_{k}^{n+1})^2\\
	&+4(\lambda_{i}^{n+1}-\lambda_{k}^{n+1})^2(1+\lambda_{i}^{n+1}\lambda_{k}^{n+1})^2]^{\frac{1}{3}}\\
	\leqslant& (2F_1)^{\frac{1}{3}}.
\end{align*}
	Since $c>\frac{24}{5}-\frac{16}{6}>2$, we get
	\[
	(\lambda_i^{n+1})^2+(\lambda_j^{n+1})^2+(\lambda_k^{n+1})^2 - 2\lambda_i^{n+1}\lambda_j^{n+1} - 2\lambda_j^{n+1}\lambda_k^{n+1} - 2\lambda_i^{n+1}\lambda_k^{n+1}
	\leqslant a_{1}+4+\sqrt[3]{cF_1}.
	\]
	For $i\ne j$, since \[n+\delta\geqslant a_{1}\geqslant n\geqslant 6 \qquad \text{and}\qquad  (\lambda_{i}^{n+1})^2+(\lambda_{j}^{n+1})^2\leqslant a_{1},\] Lemma \ref{lem:XLX-pair} gives
	\begin{align*}
	(\lambda_j^{n+1})^2-4\lambda_i^{n+1}\lambda_j^{n+1}=&(\lambda_i^{n+1})^2+(\lambda_j^{n+1})^2+4-\left(4\lambda_i^{n+1}\lambda_j^{n+1}+(\lambda_i^{n+1})^2+4\right)\\
	<&a_{1}+4+\sqrt[3]{\frac{16}{5}\left(3-\frac{10}{a_{1}}\right)(\lambda_{i}^{n+1}-\lambda_{j}^{n+1})^2(1+\lambda_{i}^{n+1}\lambda_{j}^{n+1})^2}\\
	\leqslant& a_{1}+4+\sqrt[3]{\left(\frac{24}{5}-\frac{16}{a_{1}}\right)F_1}\\
	\leqslant& a_{1}+4+\sqrt[3]{cF_1}.
	\end{align*}
	Therefore,
	\begin{equation}\label{key-plus}
	\begin{aligned}
		3(A_{n+1,\beta}-2B_{n+1,\beta})
		&\leqslant\left(a_{1}+4+\sqrt[3]{cF_1}\right)
		\left(\sum_{i,j,k\;\mathrm{distinct}}(h_{ijk}^{\beta})^2+3\sum_{i\ne j}(h_{iij}^{\beta})^2\right)\\
		&\leqslant\left(a_{1}+4+\sqrt[3]{cF_1}\right)|\nabla h^{\beta}|^2.
	\end{aligned}
\end{equation}
Moreover, denote by $K_{ij}$ the sectional curvature of the $2$-plane $e_{i}\wedge e_{j}$. By the Gauss equation,
\[
K_{ij}=1+\sum_\beta\lambda_i^\beta\lambda_j^\beta.
\]
Hence \[1+\lambda_{i}^{n+1}\lambda_{j}^{n+1}=K_{ij}-\sum_{\gamma=n+2}^{n+m}\lambda_i^\gamma\lambda_j^\gamma.\] 
Consequently, for every $\rho>0$, we have
\begin{equation}\label{F1}
\begin{aligned}
	F_1&=\sum_{i, j}(\lambda_{i}^{n+1}-\lambda_{j}^{n+1})^2(K_{ij}-\sum_{\gamma=n+2}^{n+m}\lambda_i^\gamma\lambda_j^\gamma)^2\\
	&\leqslant(1+\rho)\sum_{i,j}(\lambda_{i}^{n+1}-\lambda_{j}^{n+1})^2K_{ij}^2
	+(1+\rho^{-1})\sum_{i,j}(\lambda_{i}^{n+1}-\lambda_{j}^{n+1})^2(\sum_{\gamma=n+2}^{n+m}\lambda_i^\gamma\lambda_j^\gamma)^2.
\end{aligned}
\end{equation}
By the Cauchy inequality,
\[\big(\sum_{\gamma\geqslant n+2}\lambda_i^\gamma\lambda_j^\gamma\big)^2\leqslant\big(\sum_{\gamma\geqslant n+2}(\lambda_i^\gamma)^2\big)\big(\sum_{\gamma\geqslant n+2}(\lambda_j^\gamma)^2\big).\]
Thus \eqref{F1} becomes
\begin{equation}\label{F1F}
\begin{aligned}
	F_{1}&\leqslant(1+\rho)F+(1+\rho^{-1})\left[2b\sum_{i}(\lambda_{i}^{n+1})^2\sum_{\gamma\geqslant n+2}(\lambda_{i}^{\gamma})^2-2\bigg(\sum_{i}\lambda_{i}^{n+1}\sum_{\gamma\geqslant n+2}(\lambda_{i}^{\gamma})^2\bigg)^2\right]\\
	&\leqslant(1+\rho)F+2b(1+\rho^{-1})\sum_{\gamma\geqslant n+2}\sum_{i}(\lambda_{i}^{n+1})^2(\lambda_{i}^{\gamma})^2\\
	&\leqslant(1+\rho)F+(1+\rho^{-1})\frac{n-2}{n}(S-\epsilon)\epsilon^2,
\end{aligned}
\end{equation}
where the last step uses the sharp quartic estimate \eqref{eq:opt-three}, $b\leqslant \epsilon$, and the monotonicity of $t^2(S-t)$ on $0\leqslant t\leqslant\epsilon<\frac{S}{2}$.
Finally, \eqref{key}, \eqref{key-plus} and \eqref{F1F} yield
\begin{align*}
	&3\sum_{\alpha,\beta}(A_{\alpha,\beta}-2B_{\alpha,\beta})\\=&3\sum_{\beta}(A_{n+1,\beta}-2B_{n+1,\beta})+3\sum_{\beta}\sum_{\gamma\geqslant n+2}(A_{\gamma,\beta}-2B_{\gamma,\beta})\\
	\leqslant& \left(a_{1}+4+\sqrt[3]{cF_1}\right)|\nabla h|^2+\Gamma_n b|\nabla h|^2\\
	\leqslant&\left(S+4+(\Gamma_n-1)\epsilon+\sqrt[3]{cF_{1}}\right)|\nabla h|^2\\
	\leqslant&\left(S+4+(\Gamma_n-1)\epsilon+\sqrt[3]{c(1+\rho)F+c(1+\rho^{-1})\frac{n-2}{n}(S-\epsilon)\epsilon^2}\right)|\nabla h|^2.
\end{align*}
For $3\leqslant n\leqslant5$, the additional assumption $n+\delta<6$ gives $S\leqslant n+\delta<6$, and hence $a_{1}<6$. Therefore, for any $i,j$,
\[(\lambda_{i}^{n+1})^2+(\lambda_{j}^{n+1})^2\leqslant a_{1}<6.\]
Thus Lemma~\ref{lem:XLX-pair} can be applied below with $s=6$.
		For distinct $i,j,k$, Lemma~\ref{lem:XLX-triple} gives
	\[
		(\lambda_i^{n+1})^2+(\lambda_j^{n+1})^2+(\lambda_k^{n+1})^2 - 2\lambda_i^{n+1}\lambda_j^{n+1} - 2\lambda_j^{n+1}\lambda_k^{n+1} - 2\lambda_i^{n+1}\lambda_k^{n+1}
	\leqslant a_{1}+4+\sqrt[3]{2F_1}.
	\]
	Lemma~\ref{lem:XLX-pair} with $s=6$ gives
	\[
		-\bigl((\lambda_i^{n+1})^2+4\lambda_i^{n+1}\lambda_j^{n+1}+4\bigr)^3
	<\frac{64}{15}(\lambda_i^{n+1}-\lambda_j^{n+1})^2
	(1+\lambda_i^{n+1}\lambda_j^{n+1})^2.
	\]
	Since $F_1$ contains both ordered pairs $(i,j)$ and $(j,i)$, $F_1\geqslant2(\lambda_i^{n+1}-\lambda_j^{n+1})^2(1+\lambda_i^{n+1}\lambda_j^{n+1})^2$. Thus
	\begin{align*}
	(\lambda_j^{n+1})^2-4\lambda_i^{n+1}\lambda_j^{n+1}=&(\lambda_i^{n+1})^2+(\lambda_j^{n+1})^2+4-\left(4\lambda_i^{n+1}\lambda_j^{n+1}+(\lambda_i^{n+1})^2+4\right)\\
	\leqslant& a_1+4+\sqrt[3]{\frac{32}{15}F_1}.
\end{align*} 
Hence, by \eqref{key}, \eqref{decomp} and \eqref{F1F}, the estimate \eqref{sumA-2B} holds.
\end{proof}

We next use the following integral formula for $\sum\limits_{\alpha,\beta}(A_{\alpha,\beta}-2B_{\alpha,\beta})$, which generalizes Theorem 4 of Peng and Terng \cite{Peng-T2}.
\begin{lem}\label{AAA}
		Let $M^n$ be a closed minimal submanifold in the unit sphere $\mathbb{S}^{n+m}$ $(m\geqslant1)$ with flat normal bundle. Then
\begin{align}
		\int_{M}&\sum_{\alpha,\beta}(A_{\alpha,\beta}-2B_{\alpha,\beta})dM\label{1stoccur}\\
		=\int_{M}\bigg\lbrace&-\frac{1}{4}|\nabla S|^2+\sum_{i,\alpha,\beta,\gamma}\lambda_{i}^{\alpha}(\lambda_{i}^{\beta})^2\lambda_{i}^{\gamma}\langle A^{\alpha},A^{\gamma}\rangle-S^2-\sum_{i,k,\alpha,\beta,\gamma}(\lambda_{k}^{\alpha})^2\lambda_{k}^{\gamma}(\lambda_{i}^{\beta})^2\lambda_{i}^{\gamma}\notag\\
		&+\sum_{i, j,k,\alpha,\beta}\Big(4\lambda_{k}^{\alpha}\lambda_{i}^{\beta}h_{ijk}^{\alpha}h_{ijk}^{\beta}-\lambda_{i}^{\alpha}\lambda_{j}^{\alpha}h_{iik}^{\beta}h_{jjk}^{\beta}-\lambda_{i}^{\alpha}\lambda_{j}^{\beta}h_{jjk}^{\alpha}h_{iik}^{\beta}-2\lambda_{i}^{\alpha}\lambda_{i}^{\beta}h_{ijk}^{\alpha}h_{ijk}^{\beta}\Big)\notag\\
		&+2\Big(\sum_{i,\alpha,\beta,\gamma}\lambda_{i}^{\alpha}(\lambda_{i}^{\beta})^2\lambda_{i}^{\gamma}\langle A^{\alpha},A^{\gamma}\rangle-|\mathcal{A}|^2-\sum_{\alpha,\beta,\gamma}\big(\sum_{i}\lambda_{i}^{\alpha}\lambda_{i}^{\beta}\lambda_{i}^{\gamma}\big)^2\Big)\bigg\rbrace dM.\notag
\end{align}
\end{lem}
\begin{proof}
Since the normal bundle of $M$ is flat, near any given point $q\in M$ we can choose a local orthonormal tangent frame $\{e_{1},e_{2},\ldots,e_n\}$ and an orthonormal normal frame so that, for all $\alpha$, $h_{ij}^{\alpha}=\lambda_{i}^{\alpha}\delta_{ij}$ at $q$. Then
\begin{equation}\label{Will}
	\begin{aligned}
		&\sum_{i,j,\alpha}h_{ij}^{\alpha}\Big(\sum_{\beta}\operatorname{Tr}(A^{\alpha}A^{\beta}A^{\beta})\Big)_{ij}\\
		=&\sum_{k,\alpha}\lambda_{k}^{\alpha}\Big(\sum_{\beta,l,m,p} h_{lm}^{\alpha}h_{mp}^{\beta}h_{pl}^{\beta}\Big)_{kk}\\
		=&\sum_{k,\alpha}\lambda_{k}^{\alpha}\Big(\sum_{i,\beta}h_{iikk}^{\alpha}(\lambda_{i}^{\beta})^2+2\sum_{i, j,\beta}\lambda_{i}^{\alpha}(h_{ijk}^{\beta})^2+4\sum_{i, j,\beta}\lambda_{i}^{\beta}h_{ijk}^{\alpha}h_{ijk}^{\beta}+2\sum_{i,\beta}h_{iikk}^{\beta}\lambda_{i}^{\alpha}\lambda_{i}^{\beta}\Big).
	\end{aligned}
\end{equation}
By the Ricci formula, we have
\begin{equation}\label{Ricc}
	\begin{aligned}
	\sum_{i,k,\alpha,\beta}h_{iikk}^{\alpha}\lambda_{k}^{\alpha}(\lambda_{i}^{\beta})^2=&\sum_{i,k,\alpha,\beta}\Big(h_{kkii}^{\alpha}+(\lambda_{i}^{\alpha}-\lambda_{k}^{\alpha})(1+\sum_{\gamma}\lambda_{i}^{\gamma}\lambda_{k}^{\gamma})\Big)\lambda_{k}^{\alpha}(\lambda_{i}^{\beta})^2\\
	=&\sum_{i,\alpha,\beta}\Big(\frac{1}{2}(\sum_{p,l}(h_{pl}^{\alpha})^2)_{ii}-\sum_{j,k}(h_{ijk}^{\alpha})^2\Big)(\lambda_{i}^{\beta})^2+\sum_{i,k,\alpha,\beta,\gamma}\lambda_{i}^{\alpha}\lambda_{k}^{\alpha}(\lambda_{i}^{\beta})^2\lambda_{i}^{\gamma}\lambda_{k}^{\gamma}\\
	&-\sum_{\alpha,\beta}\|A^{\alpha}\|^2\|A^{\beta}\|^2-\sum_{i,k,\alpha,\beta,\gamma}(\lambda_{k}^{\alpha})^2(\lambda_{i}^{\beta})^2\lambda_{i}^{\gamma}\lambda_{k}^{\gamma}.
	\end{aligned}
\end{equation}
Substituting \eqref{Ricc} into \eqref{Will}, we obtain
\begin{align}
		&\sum_{i,j,\alpha}h_{ij}^{\alpha}
		\Big(\sum_{\beta}\operatorname{Tr}(A^{\alpha}A^{\beta}A^{\beta})\Big)_{ij}\label{sst}\\
		={}&-\sum_{\alpha,\beta}(A_{\alpha,\beta}-2B_{\alpha,\beta})
		+\frac{1}{2}\sum_{i,j,k,\alpha,\beta}h_{ik}^{\beta}h_{kj}^{\beta}
		\Big(\sum_{p,l}(h_{pl}^{\alpha})^2\Big)_{ij}\notag\\
		&+\sum_{i,\alpha,\beta,\gamma}\lambda_{i}^{\alpha}(\lambda_{i}^{\beta})^2\lambda_{i}^{\gamma}
		\langle A^{\alpha},A^{\gamma}\rangle-S^2\notag-\sum_{i,k,\alpha,\beta,\gamma}(\lambda_{k}^{\alpha})^2\lambda_{k}^{\gamma}
		(\lambda_{i}^{\beta})^2\lambda_{i}^{\gamma}\\
		&+4\sum_{i,j,k,\alpha,\beta}\lambda_{k}^{\alpha}\lambda_{i}^{\beta}h_{ijk}^{\alpha}h_{ijk}^{\beta}+2\sum_{i,k,\alpha,\beta}h_{iikk}^{\beta}\lambda_{k}^{\alpha}\lambda_{i}^{\alpha}\lambda_{i}^{\beta}.\notag
\end{align}
Integrating \eqref{sst} over $M$ and using Stokes' theorem and the minimality of $M$ yields
\begin{align}
		\int_{M}\sum_{\alpha,\beta}(A_{\alpha,\beta}-2B_{\alpha,\beta})dM
		=\int_{M}\bigg(&-\frac{1}{2}\sum_{i,j,k,\alpha,\beta}
		(h_{ik}^{\beta}h_{kj}^{\beta})_{j}\Big(\sum_{p,l}(h_{pl}^{\alpha})^2\Big)_{i}\label{zjb}\\
		&+\sum_{i,\alpha,\beta,\gamma}\lambda_{i}^{\alpha}(\lambda_{i}^{\beta})^2\lambda_{i}^{\gamma}
		\langle A^{\alpha},A^{\gamma}\rangle-S^2\notag\\
		&-\sum_{i,k,\alpha,\beta,\gamma}(\lambda_{k}^{\alpha})^2\lambda_{k}^{\gamma}
		(\lambda_{i}^{\beta})^2\lambda_{i}^{\gamma}
		+4\sum_{i,j,k,\alpha,\beta}\lambda_{k}^{\alpha}\lambda_{i}^{\beta}h_{ijk}^{\alpha}h_{ijk}^{\beta}\notag\\
	&+2\sum_{i,k,\alpha,\beta}h_{iikk}^{\beta}\lambda_{k}^{\alpha}\lambda_{i}^{\alpha}\lambda_{i}^{\beta}\bigg)dM.\notag
\end{align}
On the other hand,
\begin{equation}\label{Ricc2}
	\begin{aligned}
		\sum_{i,k,\alpha,\beta}h_{iikk}^{\beta}\lambda_{k}^{\alpha}\lambda_{i}^{\alpha}\lambda_{i}^{\beta}
		={}&\sum_{i,k,\alpha,\beta}
		\Big(h_{kkii}^{\beta}+(\lambda_{i}^{\beta}-\lambda_{k}^{\beta})
		\big(1+\sum_{\gamma}\lambda_{i}^{\gamma}\lambda_{k}^{\gamma}\big)\Big)
		\lambda_{k}^{\alpha}\lambda_{i}^{\alpha}\lambda_{i}^{\beta}\\
		={}&\sum_{i,k,\alpha,\beta}h_{kkii}^{\beta}\lambda_{k}^{\alpha}\lambda_{i}^{\alpha}\lambda_{i}^{\beta}
		+\sum_{i,\alpha,\beta,\gamma}\lambda_{i}^{\alpha}(\lambda_{i}^{\beta})^2\lambda_{i}^{\gamma}
		\langle A^{\alpha},A^{\gamma}\rangle\\
		&-|\mathcal{A}|^2
		-\sum_{\alpha,\beta,\gamma}\big(\sum_{i}\lambda_{i}^{\alpha}\lambda_{i}^{\beta}\lambda_{i}^{\gamma}\big)^2.
	\end{aligned}
\end{equation}
Since
\begin{align*}
	&\sum_{i,k,\alpha,\beta}h_{kkii}^{\beta}\lambda_{k}^{\alpha}\lambda_{i}^{\alpha}\lambda_{i}^{\beta}\\[-1mm]
	={}&\sum_{i,\alpha,\beta}\bigg(\Big(\sum_{p,l}h_{pl}^{\alpha}h_{pl}^{\beta}\Big)_{ii}
	-\sum_{p,l}h_{plii}^{\alpha}h_{pl}^{\beta}
	-2\sum_{j,k}h_{jki}^{\alpha}h_{jki}^{\beta}\bigg)
	\lambda_{i}^{\alpha}\lambda_{i}^{\beta}\\
	={}&\sum_{i,j,k,\alpha,\beta}\Big(\sum_{p,l}h_{pl}^{\alpha}h_{pl}^{\beta}\Big)_{ij}
	h_{ik}^{\alpha}h_{kj}^{\beta}
	-2\sum_{i,j,k,\alpha,\beta}\lambda_{i}^{\alpha}\lambda_{i}^{\beta}h_{ijk}^{\alpha}h_{ijk}^{\beta}-\sum_{i,k,\alpha,\beta}h_{kkii}^{\beta}\lambda_{k}^{\alpha}\lambda_{i}^{\alpha}\lambda_{i}^{\beta}.
\end{align*}
Hence, we have
\begin{equation}\label{yx}
	2\sum_{i,k,\alpha,\beta}h_{kkii}^{\beta}\lambda_{k}^{\alpha}\lambda_{i}^{\alpha}\lambda_{i}^{\beta}=\sum_{i,j,k,\alpha,\beta}\Big(\sum_{p,l}h_{pl}^{\alpha}h_{pl}^{\beta}\Big)_{ij}h_{ik}^{\alpha}h_{kj}^{\beta}-2\sum_{i, j,k,\alpha,\beta}\lambda_{i}^{\alpha}\lambda_{i}^{\beta}h_{ijk}^{\alpha}h_{ijk}^{\beta}.
\end{equation}
By Stokes' theorem and the minimality of $M$, integrating both sides of \eqref{yx} yields
\begin{align}
		&2\int_{M}\sum_{i,k,\alpha,\beta}h_{kkii}^{\beta}
		\lambda_{k}^{\alpha}\lambda_{i}^{\alpha}\lambda_{i}^{\beta}dM\label{int2}\\
		={}&\int_{M}\sum_{i,j,k,\alpha,\beta}\bigg(
		-h_{ikj}^{\alpha}h_{kj}^{\beta}\Big(\sum_{p,l}h_{pl}^{\alpha}h_{pl}^{\beta}\Big)_{i}
		-2\lambda_{i}^{\alpha}\lambda_{i}^{\beta}h_{ijk}^{\alpha}h_{ijk}^{\beta}\bigg)dM\notag\\
		={}&\int_{M}\sum_{i,j,k,\alpha,\beta}\Big(
		-\lambda_{i}^{\alpha}\lambda_{j}^{\alpha}h_{iik}^{\beta}h_{jjk}^{\beta}
		-\lambda_{i}^{\alpha}\lambda_{j}^{\beta}h_{jjk}^{\alpha}h_{iik}^{\beta}\notag
		-2\lambda_{i}^{\alpha}\lambda_{i}^{\beta}h_{ijk}^{\alpha}h_{ijk}^{\beta}\Big)dM.\notag
\end{align}
Combining \eqref{Ricc2} with \eqref{int2}, we obtain
\begin{equation*}
	\begin{aligned}
	2\int_{M}&\sum_{i,k,\alpha,\beta}h_{iikk}^{\beta}\lambda_{k}^{\alpha}\lambda_{i}^{\alpha}\lambda_{i}^{\beta}dM\\
	=\int_{M}&\bigg\lbrace\sum_{i,j,k,\alpha,\beta}\Big(-\lambda_{i}^{\alpha}\lambda_{j}^{\alpha}h_{iik}^{\beta}h_{jjk}^{\beta}-\lambda_{i}^{\alpha}\lambda_{j}^{\beta}h_{jjk}^{\alpha}h_{iik}^{\beta}-2\lambda_{i}^{\alpha}\lambda_{i}^{\beta}h_{ijk}^{\alpha}h_{ijk}^{\beta}\Big)\\
	&+2\Big(\sum_{i,\alpha,\beta,\gamma}\lambda_{i}^{\alpha}(\lambda_{i}^{\beta})^2\lambda_{i}^{\gamma}\langle A^{\alpha},A^{\gamma}\rangle-|\mathcal{A}|^2-\sum_{\alpha,\beta,\gamma}\big(\sum_{i}\lambda_{i}^{\alpha}\lambda_{i}^{\beta}\lambda_{i}^{\gamma}\big)^2\Big)\bigg\rbrace dM.
	\end{aligned}
\end{equation*}
Therefore, this together with \eqref{zjb} gives \eqref{1stoccur}.
\end{proof}
We also need the following algebraic lemma.
\begin{lem}\label{2026}
	Let $\mathtt{a}_{ij}^{\alpha}$, $\mathtt{b}_{i}^{\alpha}$ $(i,j=1,\ldots,n;\alpha=1,\ldots,m)$ be real numbers satisfying
	\begin{equation*}
		\begin{aligned}
			\sum_{i}\mathtt{b}_{i}^{\alpha}=0,\quad \sum_{i,\alpha}(\mathtt{b}_{i}^{\alpha})^2=\mathcal{B}(> 0),\quad \sum_{i,j,\alpha}(\mathtt{b}_{i}^{\alpha}+\mathtt{b}_{j}^{\alpha})\mathtt{a}_{ij}^{\alpha}=\mathcal{C}.
		\end{aligned}
	\end{equation*}
	Then we have
	\begin{equation*}
		\begin{aligned}
			\sum_{i,\alpha}(\mathtt{a}_{ii}^{\alpha})^{2}+3\sum_{\alpha}\sum_{i\neq j}(\mathtt{a}_{ij}^{\alpha})^{2}\geqslant \frac{3\mathcal{C}^{2}}{2(n+4)\mathcal{B}}.
		\end{aligned}
	\end{equation*}
\end{lem}
\begin{proof}
	Since
	\begin{align*}
		\mathcal{C}=\sum_{i,\alpha}2\mathtt{b}_{i}^{\alpha}\mathtt{a}_{ii}^{\alpha}+\sum_{\alpha}\sum_{i\ne j}\frac{\mathtt{b}_{i}^{\alpha}+\mathtt{b}_{j}^{\alpha}}{\sqrt{3}}\sqrt{3}\mathtt{a}_{ij}^{\alpha},
	\end{align*}
	the Cauchy inequality gives
	\begin{align}\label{2factor}
		\mathcal C^2
		\leqslant
		\Big(
		\sum_{i,\alpha}(\mathtt{a}_{ii}^{\alpha})^2
		+3\sum_{\alpha}\sum_{i\ne j}(\mathtt{a}_{ij}^{\alpha})^2
		\Big)
		\Big(
		\sum_{i,\alpha}(2\mathtt{b}_i^{\alpha})^2
		+\frac13\sum_{\alpha}\sum_{i\ne j}(\mathtt{b}_i^{\alpha}+\mathtt{b}_j^{\alpha})^2
		\Big).
	\end{align}
	For each fixed $\alpha$, since $\sum\limits_i \mathtt{b}_i^{\alpha}=0$, we have
	\[
	\sum_{i\ne j}(\mathtt{b}_i^{\alpha}+\mathtt{b}_j^{\alpha})^2
	=2(n-2)\sum_i(\mathtt{b}_i^{\alpha})^2.
	\]
	Hence, the second factor in \eqref{2factor} equals
	$\frac{2(n+4)}3\mathcal B$, which completes the proof.
\end{proof}
Finally, we provide a lower bound for $|\nabla^2 h|^2$.
\begin{lem}\label{xj}
		Let $M^n$ be a closed minimal submanifold in $\mathbb{S}^{n+m}$ with flat normal bundle. If $S=0$, then $h\equiv0$. If $S>0$, then
		\begin{equation}\label{0xj}
			\begin{aligned}
			|\nabla^2 h|^2\geqslant\frac{3}{4}F+\frac{3\left(nS-|\mathcal{A}|^2\right)^2}{2(n+4)S},
		\end{aligned}
		\end{equation}
	where $F$, defined in \eqref{F}, can be expressed as
	 \begin{align*}
		F=2\Big(\sum_{i,\alpha,\beta,\gamma}\lambda_{i}^{\alpha}(\lambda_{i}^{\beta})^2\lambda_{i}^{\gamma}\langle A^{\alpha},A^{\gamma}\rangle-\sum_{\alpha,\beta,\gamma}\big(\sum_{i}\lambda_{i}^{\alpha}\lambda_{i}^{\beta}\lambda_{i}^{\gamma}\big)^2-|\mathcal{A}|^2-(|\mathcal{A}|^2-nS)\Big).
\end{align*}
\end{lem}
\begin{proof}
	The case $S=0$ is immediate, since then $h\equiv0$. Hence we assume $S>0$.
	For each $\alpha$, define
	\begin{equation}\label{uijkl} u_{ijkl}^{\alpha}=\dfrac{1}{4}(h_{ijkl}^{\alpha}+h_{jkli}^{\alpha}+h_{klij}^{\alpha}+h_{lijk}^{\alpha}).
	\end{equation}
	Squaring \eqref{uijkl} and summing over all indices gives
	\[\sum_{i,j,k,l,\alpha}(u_{ijkl}^{\alpha})^2=\frac14\sum_{i,j,k,l,\alpha}(h_{ijkl}^{\alpha})^2+
	\frac34\sum_{i,j,k,l,\alpha}h^{\alpha}_{ijkl}h_{ijlk}^{\alpha}.
	\]
	On the other hand,
	\[
	\sum_{i,j,k,l,\alpha}(h_{ijkl}^{\alpha}-h_{ijlk}^{\alpha})^2
	=2\sum_{i,j,k,l,\alpha}(h_{ijkl}^{\alpha})^2-2\sum_{i,j,k,l,\alpha}h_{ijkl}^{\alpha}h_{ijlk}^{\alpha}.
	\]
	Consequently,
	\begin{equation}\label{yuanshi}
	\sum_{i,j,k,l,\alpha}(h_{ijkl}^{\alpha})^2-\sum_{i,j,k,l,\alpha}(u_{ijkl}^{\alpha})^2
	=
	\frac38
	\sum_{i,j,k,l,\alpha}(h_{ijkl}^{\alpha}-h_{ijlk}^{\alpha})^2.
	\end{equation}
	Because the normal bundle of $M$ is flat, the Gauss equation and the Ricci formula give 
	\[h_{ijkl}^{\alpha}-h_{ijlk}^{\alpha}=(\delta_{ik}\delta_{jl}-\delta_{il}\delta_{jk})(\lambda_{i}^{\alpha}-\lambda_{j}^{\alpha})(1+\sum_{\beta}\lambda_{i}^{\beta}\lambda_{j}^{\beta})\]
	and
	\begin{equation*}
		h_{ijij}^{\alpha}-h_{jjii}^{\alpha}=(\lambda_{i}^{\alpha}-\lambda_{j}^{\alpha})(1+\sum_{\beta}\lambda_{i}^{\beta}\lambda_{j}^{\beta})=(\lambda_{i}^{\alpha}-\lambda_{j}^{\alpha})K_{ij},
	\end{equation*}
	where $K_{ij}$ is the sectional curvature of the $2$-plane $e_{i}\wedge e_{j}$.
	Thus \eqref{yuanshi} becomes
	\begin{equation*}
		\sum_{i,j,k,l,\alpha}(h_{ijkl}^{\alpha})^2-\sum_{i,j,k,l,\alpha}(u_{ijkl}^{\alpha})^2=\frac{3}{4}\sum_{i,j,\alpha}(h_{ijij}^{\alpha}-h_{jjii}^{\alpha})^2.
	\end{equation*}
 Since the right-hand side equals the difference of two squared norms of
 globally defined tensors, $\sum\limits_{i,j,\alpha}(h_{ijij}^{\alpha}-h_{jjii}^{\alpha})^2$ is independent of the
 choice of frames.
Since
\begin{equation*}
	\begin{aligned}
		\Delta h_{ij}^{\alpha}=&nh_{ij}^{\alpha}-\sum_{\beta}h_{ij}^{\beta}\langle A^{\beta},A^{\alpha}\rangle,
	\end{aligned}
\end{equation*}
we obtain
\begin{equation*}
	\begin{aligned}	
		\sum_{i,j,\alpha}u_{iijj}^{\alpha}\lambda_{i}^{\alpha}=\frac{1}{2}\sum_{\alpha}\Big(n\|A^{\alpha}\|^2-\sum_{\beta}\langle A^{\alpha},A^{\beta}\rangle^2\Big)=\frac{1}{2}(nS-|\mathcal{A}|^2).
	\end{aligned}
\end{equation*}
Hence, Lemma~\ref{2026} gives
\begin{equation*}
	\begin{aligned}	
		\sum_{i,j,k,l,\alpha}(u_{ijkl}^{\alpha})^2\geqslant&\sum_{i,\alpha}(u_{iiii}^{\alpha})^2+3\sum_{\alpha}\sum_{i\neq j}(u_{iijj}^{\alpha})^2\\
		\geqslant&\frac{3\left(nS-|\mathcal{A}|^2\right)^2}{2(n+4)S}.
	\end{aligned}
\end{equation*}
Consequently,
	\begin{equation*}
		\begin{aligned}
			\sum_{i,j,k,l,\alpha}(h_{ijkl}^{\alpha})^2\geqslant&\frac{3}{4}\sum_{i,j,\alpha}(\lambda_{i}^{\alpha}-\lambda_{j}^{\alpha})^2(1+\sum_{\beta}\lambda_{i}^{\beta}\lambda_{j}^{\beta})^2+\frac{3\left(nS-|\mathcal{A}|^2\right)^2}{2(n+4)S}\\
			=&\frac{3}{2}\Big(\sum_{i,\alpha,\beta,\gamma}\lambda_{i}^{\alpha}(\lambda_{i}^{\beta})^2\lambda_{i}^{\gamma}\langle A^{\alpha},A^{\gamma}\rangle-\sum_{\alpha,\beta,\gamma}(\sum_{i}\lambda_{i}^{\alpha}\lambda_{i}^{\beta}\lambda_{i}^{\gamma})^2-2|\mathcal{A}|^2+nS\Big)\\
			&+\frac{3\left(nS-|\mathcal{A}|^2\right)^2}{2(n+4)S}.
		\end{aligned}
	\end{equation*}
\end{proof}

\section{Proof of Theorem~\ref{main} in codimension \texorpdfstring{$2$}{2}}\label{4}
In this section, we consider the case of codimension $m=2$.
\begin{prop}
	Let $M^n$ be a closed minimal submanifold in the unit sphere $\mathbb{S}^{n+2}$ with flat normal bundle. Then we have
	\begin{equation}\label{oriS}
		\int_{M}|\nabla h|^2dM=\int_{M}\big(-nS+|\mathcal{A}|^2\big)dM=\int_{M}\big(S(S-n)-2\det\mathcal{A}\big)dM.
	\end{equation}
\end{prop}
\begin{proof}
	Since the normal bundle of $M$ is flat, \eqref{Simons} becomes
		\begin{equation}\label{1}
			\begin{aligned}
				\frac{1}{2}\Delta S=|\nabla h|^2+nS-|\mathcal{A}|^2.
			\end{aligned}
		\end{equation}
		A direct expansion gives
		\begin{equation}\label{relation}
			\begin{aligned}
				|\mathcal{A}|^2=&\|A^{n+1}\|^4+\|A^{n+2}\|^4+2|\langle A^{n+1},A^{n+2}\rangle|^2\\
				=&S^2-(2\|A^{n+1}\|^2\|A^{n+2}\|^2-2|\langle A^{n+1},A^{n+2}\rangle|^2)\\
				=&S^2-2\det\mathcal{A}.
			\end{aligned}
		\end{equation}
	 Substituting \eqref{relation} into \eqref{1} gives
		\begin{equation}\label{1p}
			\begin{aligned}
				\frac{1}{2}\Delta S=|\nabla h|^2+nS-S^2+2\det\mathcal{A}.
			\end{aligned}
		\end{equation}
	Integrating \eqref{1} and \eqref{1p} over $M$ yields \eqref{oriS}.
\end{proof}
\begin{cor}
	Let $M^n$ be a closed minimal submanifold in the unit sphere $\mathbb{S}^{n+2}$ with flat normal bundle. If $S$ is constant, then we have
	\begin{equation}\label{SimonsS}
		|\nabla h|^2=-nS+|\mathcal{A}|^2=S(S-n)-2\det\mathcal{A}.
	\end{equation}
\end{cor}
Using refined estimates, we give an upper bound for $\int_{M}|\nabla^2 h|^2dM$.
\begin{thm}\label{upperco2}
	Let $M^n$ be a closed minimal submanifold in the unit sphere $\mathbb{S}^{n+2}$ with flat normal bundle. If $S$ is constant and $n\leqslant S\leqslant n+\delta\leqslant2n$ for some $\delta>0$,
	then
	\begin{equation}\label{newupper}
	\begin{aligned}
		\int_M |\nabla^2h|^2dM
		\leqslant \int_M\bigg\lbrace&\Big(S-2n-3+10\epsilon+
		\frac{2\epsilon^2}{n+\delta}\Big)|\nabla h|^2+3\sum_{\alpha,\beta}(A_{\alpha,\beta}-2B_{\alpha,\beta})\bigg\rbrace dM,
	\end{aligned}
\end{equation}
	where $\epsilon$ is defined in \eqref{eq:b-eps}.
\end{thm}
\begin{proof}
	The right-hand side of \eqref{eq:simons-gradient} is independent of the choice of frames. At the fixed point $q$, since $\langle A^{n+1},A^{n+2}\rangle=0$, we have $a_{r}=\|A^{n+r}\|^2$, $r=1, 2$. Then \eqref{eq:simons-gradient} becomes
	\begin{equation}\label{2h}
		\begin{aligned}
			\frac{1}{2}\Delta(|\nabla h|^2)=&(2n+3-S)|\nabla h|^2-3\sum_{\alpha,\beta}(A_{\alpha,\beta}-2B_{\alpha,\beta})+a_{1}|\nabla h^{n+2}|^2\\&+a_{2}|\nabla h^{n+1}|^2-6\sum_{k}(\sum_{i}\lambda_{i}^{n+1}h_{iik}^{n+2})^2-6\sum_{k}(\sum_{i}\lambda_{i}^{n+2}h_{iik}^{n+1})^2\\
			&-6\sum_{k}(\sum_{i}\lambda_{i}^{n+1}h_{iik}^{n+1})^2-6\sum_{k}(\sum_{i}\lambda_{i}^{n+2}h_{iik}^{n+2})^2+|\nabla^2 h|^2.
		\end{aligned}
	\end{equation}
	Since $S$ is constant, 
	\begin{equation}\label{S_k}
		\frac{1}{2}S_{k}=\sum_{i}\lambda_{i}^{n+1}h_{iik}^{n+1}+\sum_{i}\lambda_{i}^{n+2}h_{iik}^{n+2}=0.
	\end{equation}
	From Lemma~\ref{fine} and \eqref{S_k}, \eqref{2h} becomes
	\begin{equation}\label{www}
	\begin{aligned}
		\frac{1}{2}\Delta(|\nabla h|^2)
		={}&(2n+3-S)|\nabla h|^2
		-3\sum_{\alpha,\beta}(A_{\alpha,\beta}-2B_{\alpha,\beta})+|\nabla^2 h|^2\\
		&+a_{1}|\nabla h^{n+2}|^2+a_{2}|\nabla h^{n+1}|^2
		-12\sum_{k}\Big(\sum_{i}\lambda_{i}^{n+2}h_{iik}^{n+1}\Big)^2\\
		&-12\sum_{k}\Big(\sum_{i}\lambda_{i}^{n+2}h_{iik}^{n+2}\Big)^2.
	\end{aligned}
\end{equation}
	By the Cauchy inequality, Lemma~\ref{fine} and \eqref{S_k}, we have
	\[
	\sum_{k}(\sum_{i}\lambda_{i}^{n+2}h_{iik}^{n+1})^2\leqslant \min\{a_2|\nabla h^{n+1}|^2,a_1|\nabla h^{n+2}|^2\}\]
	and
	\[
	\sum_{k}(\sum_{i}\lambda_{i}^{n+2}h_{iik}^{n+2})^2\leqslant \min\{a_2|\nabla h^{n+2}|^2,a_1|\nabla h^{n+1}|^2\}.
	\]
	If $a_2=0$, the desired estimate \eqref{newupper} follows immediately.
	Assume $a_2>0$ at this fixed point $q$ and put
	$w=\frac{a_{1}}{a_{2}}\geqslant1$. We claim that
	\begin{equation}\label{min-ineq}
		\begin{aligned}
			&12\min\{|\nabla h^{n+1}|^2,w|\nabla h^{n+2}|^2\}+12\min\{|\nabla h^{n+2}|^2,w|\nabla h^{n+1}|^2\}\\
			\leqslant&w|\nabla h^{n+2}|^2+|\nabla h^{n+1}|^2+\left(10+\frac{2}{w+1}\right)(|\nabla h^{n+1}|^2+|\nabla h^{n+2}|^2).
		\end{aligned}
	\end{equation}
	When $|\nabla h^{n+1}|^2>0$, put $t=\frac{|\nabla h^{n+2}|^2}{|\nabla h^{n+1}|^2}$ and
	$C=10+\frac{2}{w+1}$.  After division by $|\nabla h^{n+1}|^2$, the right-hand side minus the
	left-hand side of \eqref{min-ineq} is
	\[
	\begin{cases}
		\dfrac{11w+13}{w+1}(1-wt),&0\leqslant t\leqslant\frac{1}{w},\\[2mm]
		\dfrac{w-1}{w+1}(wt-1),&\frac{1}{w}\leqslant t\leqslant w,\\[2mm]
		(w+C)(t-w)+(w-1)^2,&t\geqslant w.
	\end{cases}
	\]
	Each expression is nonnegative.  The case $|\nabla h^{n+1}|^2=0$ follows directly, proving \eqref{min-ineq}. Therefore
	\begin{align}
			&a_{1}|\nabla h^{n+2}|^2+a_{2}|\nabla h^{n+1}|^2-12\sum_{k}(\sum_{i}\lambda_{i}^{n+2}h_{iik}^{n+1})^2-12\sum_{k}(\sum_{i}\lambda_{i}^{n+2}h_{iik}^{n+2})^2\label{ria}\\
			\geqslant&a_{1}|\nabla h^{n+2}|^2+a_{2}|\nabla h^{n+1}|^2-12\min\{a_2|\nabla h^{n+1}|^2,a_1|\nabla h^{n+2}|^2\}\notag\\
			&-12\min\{a_2|\nabla h^{n+2}|^2,a_1|\nabla h^{n+1}|^2\}\notag\\
			=&a_2(w|\nabla h^{n+2}|^2+|\nabla h^{n+1}|^2)-12a_2\min\{|\nabla h^{n+1}|^2,w|\nabla h^{n+2}|^2\}\notag\\
			&-12a_2\min\{|\nabla h^{n+2}|^2,w|\nabla h^{n+1}|^2\}\notag\\
			\geqslant&-\Big(10a_2+\frac{2a_2}{w+1}\Big)(|\nabla h^{n+1}|^2+|\nabla h^{n+2}|^2)\notag\\
			=&-\Big(10a_{2}+\frac{2a_{2}^2}{S}\Big)|\nabla h|^2.\notag
\end{align}
	Thus the pointwise inequality \eqref{ria} holds in all cases. This together with \eqref{www} yields
	\begin{equation}\label{12e}
		\frac{1}{2}\Delta(|\nabla h|^2)\geqslant\left(2n+3-S-\Big(10a_{2}+\frac{2a_{2}^2}{S}\Big)\right)|\nabla h|^2-3\sum_{\alpha,\beta}(A_{\alpha,\beta}-2B_{\alpha,\beta})+|\nabla^2 h|^2.
	\end{equation}
	Furthermore, by \eqref{eq:b-eps}, we have
	\begin{equation*}
		0\leqslant a_{2}\leqslant\frac{S}{2}-\frac{1}{2}\sqrt{S(2n-S)}=:g(S).
	\end{equation*}
	For fixed $S$, the function $a_{2}\mapsto10a_{2}+\frac{2a_{2}^2}{S}$ is increasing on
	$[0,g(S)]$, so 
	\[
	10a_2+\frac{2a_2^2}{S}
	\leqslant 10g(S)+\frac{2g(S)^2}{S}.
	\]
	Write
	\[
	\phi(S):=\frac{g(S)}{S}
	=\frac{1-\sqrt{\frac{2n}{S}-1}}2.
	\]
	Both $S$ and $\phi(S)$ are increasing on $[n,n+\delta]$, and
	$10\phi+2\phi^2$ is increasing for $\phi\geqslant0$.  Hence
	$S(10\phi(S)+2\phi(S)^2)$ is increasing and
	\[10a_2+\frac{2a_2^2}{S}\leqslant10\epsilon+\frac{2\epsilon^2}{n+\delta}.\]
	Substituting this into \eqref{12e} and integrating over $M$ yields \eqref{newupper}.
\end{proof}
Next, we specialize Lemma~\ref{AAA} to codimension two.
\begin{lem}
	Let $M^n$ be a closed minimal submanifold in the unit sphere $\mathbb{S}^{n+2}$ with flat normal bundle. Then
	\begin{align}
		&\int_{M}3\sum_{\alpha,\beta}(A_{\alpha,\beta}-2B_{\alpha,\beta})dM\label{A-2B}\\
		=\int_{M}&\bigg\lbrace3\Big(-\frac{1}{4}|\nabla S|^2
		+\sum_{i,\alpha,\beta,\gamma}\lambda_{i}^{\alpha}(\lambda_{i}^{\beta})^2\lambda_{i}^{\gamma}
		\langle A^{\alpha},A^{\gamma}\rangle\notag-|\mathcal{A}|^2-\sum_{\alpha,\beta,\gamma}\big(\sum_{i}\lambda_{i}^{\alpha}\lambda_{i}^{\beta}\lambda_{i}^{\gamma}\big)^2\Big)\\&+2\sum_{\alpha\neq\beta}(A_{\alpha,\beta}-2B_{\alpha,\beta})
		+\sum_{i,j,k}\Big(8\lambda_{i}^{n+1}\lambda_{j}^{n+2}h_{ijk}^{n+1}h_{ijk}^{n+2}
		-4\lambda_{i}^{n+1}\lambda_{i}^{n+2}h_{ijk}^{n+1}h_{ijk}^{n+2}\Big)\notag\\
		&+|\mathcal{A}|^2-S^2
		+2\Big(\sum_{i}(\lambda_{i}^{n+1})^2\lambda_{i}^{n+2}\Big)^2
		+2\Big(\sum_{i}(\lambda_{i}^{n+2})^2\lambda_{i}^{n+1}\Big)^2\notag\\
		&-2\sum_{k}(\lambda_{k}^{n+1})^3\sum_{i}(\lambda_{i}^{n+2})^2\lambda_{i}^{n+1}
		-2\sum_{k}(\lambda_{k}^{n+2})^3\sum_{i}(\lambda_{i}^{n+1})^2\lambda_{i}^{n+2}\notag\\
		&+4\sum_{i,j,k}\lambda_{i}^{n+1}\lambda_{j}^{n+2}h_{iik}^{n+1}h_{jjk}^{n+2}
		-\sum_{i,j,k}\lambda_{i}^{n+1}\lambda_{j}^{n+1}h_{iik}^{n+2}h_{jjk}^{n+2}\notag\\
		&-\sum_{i,j,k}\lambda_{i}^{n+2}\lambda_{j}^{n+2}h_{iik}^{n+1}h_{jjk}^{n+1}
		-2\sum_{i,j,k}\lambda_{i}^{n+1}\lambda_{j}^{n+2}h_{jjk}^{n+1}h_{iik}^{n+2}\bigg\rbrace dM.\notag
\end{align}
\end{lem}
\begin{proof}
Since $n+1\leqslant\alpha,\beta\leqslant n+2$, a direct expansion yields
\begin{equation}\label{1par}
	\begin{aligned}
		&\sum_{i, j,k,\alpha,\beta}(4\lambda_{k}^{\alpha}\lambda_{i}^{\beta}h_{ijk}^{\alpha}h_{ijk}^{\beta}-2\lambda_{i}^{\alpha}\lambda_{i}^{\beta}h_{ijk}^{\alpha}h_{ijk}^{\beta})\\
		=&-2\sum_{\alpha,\beta}(A_{\alpha,\beta}-2B_{\alpha,\beta})+2\sum_{\alpha\neq\beta}(A_{\alpha,\beta}-2B_{\alpha,\beta})\\
		&+\sum_{i,j,k}\left(8\lambda_{i}^{n+1}\lambda_{j}^{n+2}h_{ijk}^{n+1}h_{ijk}^{n+2}-4\lambda_{i}^{n+1}\lambda_{i}^{n+2}h_{ijk}^{n+1}h_{ijk}^{n+2}\right).
	\end{aligned}
\end{equation}
Substituting \eqref{1par} into \eqref{1stoccur} yields
\begin{align}
		&\int_{M}3\sum_{\alpha,\beta}(A_{\alpha,\beta}-2B_{\alpha,\beta})dM\label{middle}\\
		=\int_{M}&\bigg\lbrace3\Big(-\frac{1}{4}|\nabla S|^2+\sum_{i,\alpha,\beta,\gamma}\lambda_{i}^{\alpha}(\lambda_{i}^{\beta})^2\lambda_{i}^{\gamma}\langle A^{\alpha},A^{\gamma}\rangle-|\mathcal{A}|^2-\sum_{\alpha,\beta,\gamma}(\sum_{i}\lambda_{i}^{\alpha}\lambda_{i}^{\beta}\lambda_{i}^{\gamma})^2\Big)\notag\\
		&+2\sum_{\alpha\neq\beta}(A_{\alpha,\beta}-2B_{\alpha,\beta})+\sum_{i,j,k}\big(8\lambda_{i}^{n+1}\lambda_{j}^{n+2}h_{ijk}^{n+1}h_{ijk}^{n+2}-4\lambda_{i}^{n+1}\lambda_{i}^{n+2}h_{ijk}^{n+1}h_{ijk}^{n+2}\big)\notag\\
		&+|\mathcal{A}|^2-S^2+\sum_{\alpha,\beta,\gamma}\big(\sum_{i}\lambda_{i}^{\alpha}\lambda_{i}^{\beta}\lambda_{i}^{\gamma}\big)^2-\sum_{i,k,\alpha,\beta,\gamma}(\lambda_{k}^{\alpha})^2\lambda_{k}^{\gamma}(\lambda_{i}^{\beta})^2\lambda_{i}^{\gamma}\notag\\
		&+\frac{1}{2}|\nabla S|^2-\sum_{i,j,k,\alpha,\beta}\big(\lambda_{i}^{\alpha}\lambda_{j}^{\alpha}h_{iik}^{\beta}h_{jjk}^{\beta}+\lambda_{i}^{\alpha}\lambda_{j}^{\beta}h_{jjk}^{\alpha}h_{iik}^{\beta}\big)\bigg\rbrace dM.\notag
\end{align}
Besides, 
\begin{equation}\label{2par}
	\begin{aligned}
	&\sum_{\alpha,\beta,\gamma}\big(\sum_{i}\lambda_{i}^{\alpha}\lambda_{i}^{\beta}\lambda_{i}^{\gamma}\big)^2-\sum_{i,k,\alpha,\beta,\gamma}(\lambda_{k}^{\alpha})^2\lambda_{k}^{\gamma}(\lambda_{i}^{\beta})^2\lambda_{i}^{\gamma}\\
	=&
	2\big(\sum_{i}(\lambda_{i}^{n+1})^2\lambda_{i}^{n+2}\big)^2+2\big(\sum_{i}(\lambda_{i}^{n+2})^2\lambda_{i}^{n+1}\big)^2\\
	&-2\sum_{k}(\lambda_{k}^{n+1})^3\sum_{i}(\lambda_{i}^{n+2})^2\lambda_{i}^{n+1}-2\sum_{k}(\lambda_{k}^{n+2})^3\sum_{i}(\lambda_{i}^{n+1})^2\lambda_{i}^{n+2},
\end{aligned}
\end{equation}
Moreover,
\begin{equation}\label{3part}
	\begin{aligned}
		&\frac{1}{2}|\nabla S|^2
		-\sum_{i,j,k,\alpha,\beta}
		(\lambda_{i}^{\alpha}\lambda_{j}^{\alpha}h_{iik}^{\beta}h_{jjk}^{\beta}
		+\lambda_{i}^{\alpha}\lambda_{j}^{\beta}h_{jjk}^{\alpha}h_{iik}^{\beta})\\
		={}&\sum_{i,j,k}\big(4\lambda_{i}^{n+1}\lambda_{j}^{n+2}h_{iik}^{n+1}h_{jjk}^{n+2}
		-\lambda_{i}^{n+1}\lambda_{j}^{n+1}h_{iik}^{n+2}h_{jjk}^{n+2}\\
		&\qquad -\lambda_{i}^{n+2}\lambda_{j}^{n+2}h_{iik}^{n+1}h_{jjk}^{n+1}
		-2\lambda_{i}^{n+1}\lambda_{j}^{n+2}h_{jjk}^{n+1}h_{iik}^{n+2}\big).
	\end{aligned}
\end{equation}
Substituting \eqref{2par} and \eqref{3part} into \eqref{middle}, we obtain \eqref{A-2B}.
\end{proof}
When $S>0$, combining \eqref{A-2B}, Lemma~\ref{xj} and \eqref{oriS}, we obtain 
		\begin{equation}\label{geq}
			\begin{aligned}
				\int_{M}|\nabla^2 h|^2dM\geqslant\int_{M}\Big(\frac{3}{2}\sum_{\alpha,\beta}(A_{\alpha,\beta}-2B_{\alpha,\beta})+\Lambda+\frac{3}{8}|\nabla S|^2+\frac{3\left(nS-|\mathcal{A}|^2\right)^2}{2(n+4)S}\Big)dM,
			\end{aligned}
		\end{equation}
where
\begin{align}
		\Lambda={}&-\frac{3}{2}|\nabla h|^2
		-\sum_{\alpha\neq\beta}(A_{\alpha,\beta}-2B_{\alpha,\beta})
		+\sum_{k}(\lambda_{k}^{n+1})^3\sum_{i}(\lambda_{i}^{n+2})^2\lambda_{i}^{n+1}\label{Lamb}\\
		&+\sum_{k}(\lambda_{k}^{n+2})^3\sum_{i}(\lambda_{i}^{n+1})^2\lambda_{i}^{n+2}
		-\sum_{i,j,k}\bigl(4\lambda_{i}^{n+1}\lambda_{j}^{n+2}h_{ijk}^{n+1}h_{ijk}^{n+2}\notag\\
		& -2\lambda_{i}^{n+1}\lambda_{i}^{n+2}h_{ijk}^{n+1}h_{ijk}^{n+2}\bigr)
		-\frac{1}{2}|\mathcal{A}|^2+\frac{1}{2}S^2\notag\\
		&-\big(\sum_{i}(\lambda_{i}^{n+1})^2\lambda_{i}^{n+2}\big)^2
		-\big(\sum_{i}(\lambda_{i}^{n+2})^2\lambda_{i}^{n+1}\big)^2\notag\\
		&-\frac{1}{2}\sum_{i,j,k}\bigl(4\lambda_{i}^{n+1}\lambda_{j}^{n+2}h_{iik}^{n+1}h_{jjk}^{n+2}
		-\lambda_{i}^{n+1}\lambda_{j}^{n+1}h_{iik}^{n+2}h_{jjk}^{n+2}\notag\\
		&-\lambda_{i}^{n+2}\lambda_{j}^{n+2}h_{iik}^{n+1}h_{jjk}^{n+1}
		-2\lambda_{i}^{n+1}\lambda_{j}^{n+2}h_{jjk}^{n+1}h_{iik}^{n+2}\bigr).\notag
\end{align}
We next prove a pointwise estimate for $\Lambda$.
\begin{lem}\label{Lambdalemma}
Let $M^n$ be a closed minimal submanifold in the unit sphere $\mathbb{S}^{n+2}$ with flat normal bundle. If $S$ is constant and $n\leqslant S\leqslant n+\delta\leqslant2n$ for some $\delta>0$, then
\begin{equation}\label{pointwise}
	\begin{aligned}
			-\Lambda	\leqslant&\left(\frac{3}{2}+\big(3\mu_{1}+\max\{\mu_{1},\mu_{2}\}\big)\epsilon\right)|\nabla h|^2+\frac{1}{4}\mu_{1}\Delta(S^2-|\mathcal{A}|^2)\\
			&+\left(\mu_{1}(S-2n)+\frac{3n-4}{2n}S-1\right)\det\mathcal{A},
	\end{aligned}
\end{equation}
where  
\begin{equation}\label{mu}
	\begin{aligned}
	\mu_{1}=&\frac{\Gamma_n}{3}+\frac{n-1+\sqrt{n(n-2)}}{n}x+\frac{n-1}{2n}y,\\
\mu_{2}=&\frac{\Gamma_n}{3}+\frac{n-1+\sqrt{n(n-2)}}{n}\frac{1}{x}+\frac{n-1}{2n}\frac{1}{y},
\end{aligned}
\end{equation}
$x,y$ are any positive numbers, $\Gamma_n$ is defined in \eqref{eq:Gamma-n}, and $\Lambda$ is defined in \eqref{Lamb}.
\end{lem}
\begin{proof}
First, \eqref{relation} gives
 \begin{equation*}
	|\mathcal{A}|^2=S^2-2\det\mathcal{A}.
\end{equation*}
At any fixed point, choose the frame as in \eqref{eq:diag}, thus $a_{r}=\|A^{n+r}\|^2$, $r=1,2$.
By the Cauchy inequality and \eqref{eq:opt-three}, we have
\begin{equation}\label{sharper1}
\big(\sum_{i}(\lambda_{i}^{n+1})^2\lambda_{i}^{n+2}\big)^2+\big(\sum_{i}(\lambda_{i}^{n+2})^2\lambda_{i}^{n+1}\big)^2\leqslant S\sum_{i}(\lambda_{i}^{n+1})^2(\lambda_{i}^{n+2})^2\leqslant \frac{n-2}{2n}S\det\mathcal{A}.
\end{equation}
Set
\[
L_1:=\max_i|\lambda_i^{n+1}|,
\qquad
L_2:=\max_i|\lambda_i^{n+2}|.
\]
Then, by \eqref{use},
\begin{align}
	&-\sum_{k}(\lambda_{k}^{n+1})^3\sum_{i}(\lambda_{i}^{n+2})^2\lambda_{i}^{n+1}-\sum_{k}(\lambda_{k}^{n+2})^3\sum_{i}(\lambda_{i}^{n+1})^2\lambda_{i}^{n+2}\label{sharper2}\\
	\leqslant&L_1^2\sum_{k}(\lambda_{k}^{n+1})^2\sum_{i}(\lambda_{i}^{n+2})^2+L_2^2\sum_{k}(\lambda_{k}^{n+2})^2\sum_{i}(\lambda_{i}^{n+1})^2\notag\\
	\leqslant&\frac{n-1}{n}a_{1}^2a_{2}+\frac{n-1}{n}a_{2}^2a_{1}\notag\\
	=&\frac{n-1}{n}S\det\mathcal{A}.\notag
\end{align}
By \eqref{flatness-plus} and \eqref{S_k}, we have
\begin{equation}\label{wy}
	\begin{aligned}
		&\sum_{i,j,k}\bigl(4\lambda_{i}^{n+1}\lambda_{j}^{n+2}h_{iik}^{n+1}h_{jjk}^{n+2}
		-\lambda_{i}^{n+1}\lambda_{j}^{n+1}h_{iik}^{n+2}h_{jjk}^{n+2}\\
		&\qquad -\lambda_{i}^{n+2}\lambda_{j}^{n+2}h_{iik}^{n+1}h_{jjk}^{n+1}
		-2\lambda_{i}^{n+1}\lambda_{j}^{n+2}h_{jjk}^{n+1}h_{iik}^{n+2}\bigr)\\
		={}&-4\sum_{k}\Big(\sum_{i}\lambda_{i}^{n+2}h_{iik}^{n+2}\Big)^2
		-4\sum_{k}\Big(\sum_{i}\lambda_{i}^{n+2}h_{iik}^{n+1}\Big)^2\leqslant0.
	\end{aligned}
\end{equation}
Let $x,y$ be any positive numbers. From \eqref{eq:opt-same-index}, \eqref{eq:opt-two} and the Cauchy inequality, we conclude that
\begin{align}
		\big|\sum_{i,j,k}\lambda_{i}^{n+1}\lambda_{j}^{n+2}h_{ijk}^{n+1}h_{ijk}^{n+2}\big|\leqslant&\sum_{i,j,k}|\lambda_{i}^{n+1}||\lambda_{j}^{n+2}||h_{ijk}^{n+1}||h_{ijk}^{n+2}|\label{inequ1}\\
		\leqslant&\frac{n-1+\sqrt{n(n-2)}}{2n}\sum_{i,j,k}\|A^{n+1}\|\|A^{n+2}\||h_{ijk}^{n+1}||h_{ijk}^{n+2}|\notag\\
		\leqslant&\frac{n-1+\sqrt{n(n-2)}}{2n}\sqrt{\sum_{i,j,k}a_{1}(h_{ijk}^{n+2})^2\sum_{i,j,k}a_{2}(h_{ijk}^{n+1})^2}\notag\\
		\leqslant&\frac{n-1+\sqrt{n(n-2)}}{2n}\left(\frac{x}{2}a_{1}|\nabla h^{n+2}|^2+\frac{1}{2x}a_{2}|\nabla h^{n+1}|^2\right).\notag
\end{align}
Similarly, using \eqref{eq:opt-same-index}, we have
\begin{equation}\label{inequ2}
	\big|\sum_{i,j,k}\lambda_{i}^{n+1}\lambda_{i}^{n+2}h_{ijk}^{n+1}h_{ijk}^{n+2}\big|\leqslant\frac{n-1}{2n}\left(\frac{y}{2}a_{1}|\nabla h^{n+2}|^2+\frac{1}{2y}a_{2}|\nabla h^{n+1}|^2\right).
\end{equation}
Combining (\ref{inequ1}) and \eqref{inequ2}, we obtain
\begin{equation}\label{sharper3}
	\begin{aligned}
	&\sum_{i,j,k}\Big(4\lambda_{i}^{n+1}\lambda_{j}^{n+2}h_{ijk}^{n+1}h_{ijk}^{n+2}-2\lambda_{i}^{n+1}\lambda_{i}^{n+2}h_{ijk}^{n+1}h_{ijk}^{n+2}\Big)\\
	\leqslant& \Big(\frac{n-1+\sqrt{n(n-2)}}{n}x+\frac{n-1}{2n}y\Big)a_{1}|\nabla h^{n+2}|^2\\
	&+\Big(\frac{n-1+\sqrt{n(n-2)}}{n}\frac{1}{x}+\frac{n-1}{2n}\frac{1}{y}\Big)a_{2}|\nabla h^{n+1}|^2.
\end{aligned}
\end{equation}
Moreover, \eqref{key} yields
\begin{equation}\label{same}
	\sum_{\alpha\neq\beta}(A_{\alpha,\beta}-2B_{\alpha,\beta})\leqslant \frac{\Gamma_{n}}{3}(a_{1}|\nabla h^{n+2}|^2+a_{2}|\nabla h^{n+1}|^2).
\end{equation}
Hence, \eqref{sharper1}, \eqref{sharper2}, \eqref{wy}, \eqref{sharper3} and \eqref{same} together with \eqref{Lamb} give
\begin{equation}\label{zhudian}
	-\Lambda\leqslant\frac{3}{2}|\nabla h|^2+\mu_{1}a_{1}|\nabla h^{n+2}|^2+\mu_{2}a_{2}|\nabla h^{n+1}|^2+\big(\frac{3n-4}{2n}S-1\big)\det\mathcal{A},
\end{equation}
where $\mu_{1}$ and $\mu_{2}$ are defined in \eqref{mu}.
We next compute the Laplacian of $S^2-|\mathcal{A}|^2$ as follows. 
\begin{equation}\label{product}
	\begin{aligned}
		\frac{1}{4}\Delta(S^2-|\mathcal{A}|^2)
		={}&\big(S|\nabla h|^2+nS^2-S|\mathcal{A}|^2
		+\frac{1}{2}|\nabla S|^2\big)\\
		&-\Big(n|\mathcal{A}|^2-\operatorname{Tr}(\mathcal{A}^3)
		+\frac{1}{2}\sum_{\alpha,\beta}|\nabla\langle A^{\alpha},A^{\beta}\rangle|^2\\
		&\qquad +\sum_{\alpha,\beta}\langle A^{\alpha},A^{\beta}\rangle
		\sum_{k}\langle\nabla_{e_k}A^{\alpha},\nabla_{e_k}A^{\beta}\rangle\Big).
	\end{aligned}
\end{equation}
Since $\langle A^{n+1},A^{n+2}\rangle=0$ at the chosen point, (\ref{product}) becomes 
\begin{align}
		\frac{1}{4}\Delta(S^2-|\mathcal{A}|^2)
		=&a_{2}|\nabla h^{n+1}|^2+a_{1}|\nabla h^{n+2}|^2
		+(2n-S)\det\mathcal{A}+\frac{1}{2}|\nabla S|^2\label{xldr}\\
		&-2\sum_{k}(\sum_{i}\lambda_{i}^{n+1}h_{iik}^{n+1})^2-2\sum_{k}(\sum_{i}\lambda_{i}^{n+2}h_{iik}^{n+2})^2
		-|\nabla\langle A^{n+1},A^{n+2}\rangle|^2\notag\\
		=&a_{2}|\nabla h^{n+1}|^2+a_{1}|\nabla h^{n+2}|^2
		+(2n-S)\det\mathcal{A}\notag\\
		&+4\sum_{i,j,k}\lambda_{i}^{n+1}\lambda_{j}^{n+2}h_{iik}^{n+1}h_{jjk}^{n+2}
		-|\nabla\langle A^{n+1},A^{n+2}\rangle|^2.\notag
\end{align}
A direct expansion gives
\begin{equation*}
	\begin{aligned}
		|\nabla \langle A^{n+1},A^{n+2}\rangle|^2
		=\sum_{i,j,k}\big(&\lambda_{i}^{n+1}\lambda_{j}^{n+1}h_{iik}^{n+2}h_{jjk}^{n+2}
		+\lambda_{i}^{n+2}\lambda_{j}^{n+2}h_{iik}^{n+1}h_{jjk}^{n+1}\\
		&+2\lambda_{i}^{n+1}\lambda_{j}^{n+2}h_{iik}^{n+2}h_{jjk}^{n+1}\big),
	\end{aligned}
\end{equation*}
Lemma~\ref{fine} and \eqref{S_k} then transform \eqref{xldr} into
\begin{equation}\label{14L}
	\begin{aligned}
		\frac{1}{4}\Delta(S^2-|\mathcal{A}|^2)=&a_{1}|\nabla h^{n+2}|^2+a_{2}|\nabla h^{n+1}|^2+(2n-S)\det\mathcal{A}\\
		&-4\sum_{k}\big(\sum_{i}\lambda_{i}^{n+2}h_{iik}^{n+2}\big)^2-4\sum_{k}\big(\sum_{i}\lambda_{i}^{n+2}h_{iik}^{n+1}\big)^2.
	\end{aligned}
\end{equation}
Combining \eqref{zhudian} with \eqref{14L} yields
\begin{align*}
-\Lambda
\leqslant{}&\frac{3}{2}|\nabla h|^2
+\mu_{1}\Big(\frac{1}{4}\Delta(S^2-|\mathcal{A}|^2)-a_{2}|\nabla h^{n+1}|^2
+(S-2n)\det\mathcal{A}\\
&+4\sum_{k}\big(\sum_{i}\lambda_{i}^{n+2}h_{iik}^{n+2}\big)^2
+4\sum_{k}\big(\sum_{i}\lambda_{i}^{n+2}h_{iik}^{n+1}\big)^2\Big)\\
&+\mu_{2}a_{2}|\nabla h^{n+1}|^2+
\Big(\frac{3n-4}{2n}S-1\Big)\det\mathcal{A}\\
\leqslant{}&\frac{3}{2}|\nabla h|^2+\frac{1}{4}\mu_{1}\Delta(S^2-|\mathcal{A}|^2)
+(3\mu_{1}+\mu_{2})a_{2}|\nabla h^{n+1}|^2\\
&+4\mu_{1}a_{2}|\nabla h^{n+2}|^2
+\Big(\mu_{1}(S-2n)+\frac{3n-4}{2n}S-1\Big)\det\mathcal{A}\\
\leqslant{}&\Big(\frac{3}{2}+\big(3\mu_{1}+\max\{\mu_{1},\mu_{2}\}\big)\epsilon\Big)|\nabla h|^2
+\frac{1}{4}\mu_{1}\Delta(S^2-|\mathcal{A}|^2)\\
&+\Big(\mu_{1}(S-2n)+\frac{3n-4}{2n}S-1\Big)\det\mathcal{A}.
\end{align*}
\end{proof}
We now derive a lower bound for $\int_{M}|\nabla^2 h|^2dM$.
\begin{thm}\label{thmlower1}
	Let $M^n$ be a closed minimal submanifold in the unit sphere $\mathbb{S}^{n+2}$ with flat normal bundle. If $S$ is constant and $n\leqslant S\leqslant n+\delta\leqslant2n$ for some $\delta>0$, then
\begin{equation}\label{newlower}
	\begin{aligned}
		\int_{M}|\nabla^2 h|^2dM\geqslant\int_{M}&\bigg\lbrace\sigma\sum_{\alpha,\beta}(A_{\alpha,\beta}-2B_{\alpha,\beta})-\left(\sigma+\frac{2\sigma}{3}(3\mu_{1}+\max\{\mu_{1},\mu_{2}\})\epsilon\right)|\nabla h|^2\\
		&-\frac{2\sigma}{3}\left(\mu_{1}(S-2n)+\frac{3n-4}{2n}S-1\right)\det\mathcal{A}+(\frac{3}{4}-\frac{\sigma}{2})F\bigg\rbrace dM,
	\end{aligned}
\end{equation}
where $0<\sigma<\dfrac{3}{2}$, $\mu_1$ and $\mu_{2}$ are defined in \eqref{mu}.
\end{thm}
\begin{proof}
Integrating \eqref{pointwise} over $M$ yields
\begin{equation}\label{jfLmbda}
	\begin{aligned}
		\int_{M}\Lambda dM
		\geqslant\int_{M}\bigg\lbrace&-\left(\frac{3}{2}+(3\mu_{1}+\max\{\mu_{1},\mu_{2}\})\epsilon\right)|\nabla h|^2\\
		&-\left(\mu_{1}(S-2n)+\frac{3n-4}{2n}S-1\right)\det\mathcal{A}\bigg\rbrace dM.
	\end{aligned}
\end{equation}
Let $0<\sigma<\dfrac{3}{2}$. Taking the convex combination of \eqref{0xj} and \eqref{geq} with weights $1-\frac{2\sigma}{3}$ and $\frac{2\sigma}{3}$ gives
\begin{equation}\label{combination}
	\begin{aligned}
		\int_{M}|\nabla^2 h|^2dM
		\geqslant\int_{M}\bigg\lbrace&\sigma\sum_{\alpha,\beta}(A_{\alpha,\beta}-2B_{\alpha,\beta})
		+\frac{2\sigma}{3}\Lambda+\frac{\sigma}{4}|\nabla S|^2\\
		&+\Big(\frac{3}{4}-\frac{\sigma}{2}\Big)F
		+\frac{3\left(nS-|\mathcal{A}|^2\right)^2}{2(n+4)S}\bigg\rbrace dM.
	\end{aligned}
\end{equation}
Since
\[
|\nabla S|^2=0,\qquad \frac{3\left(nS-|\mathcal{A}|^2\right)^2}{2(n+4)S}\geqslant0,
\]
substituting \eqref{jfLmbda} into \eqref{combination} yields \eqref{newlower}.
\end{proof}
For further use, we need the following lemma.
\begin{lem}\label{Young}
	Assume the hypotheses of Lemma~\ref{mixA-2B}. Let $0<\sigma<3$, $\rho>0$ and set
	\begin{equation}\label{theta}
		\theta=\left(1-\frac\sigma3\right)\left(c(1+\rho)\right)^{1/3}.
	\end{equation}
	Then for any $\kappa>0$, if $S\leqslant n+\delta$, then
	\begin{equation}\label{eq:young}
		\theta F^{\frac{1}{3}}|\nabla h|^2
		\leqslant \kappa F+\frac{2}{3\sqrt3}\theta^{\frac{3}{2}}\kappa^{-\frac{1}{2}}\sqrt{(n+\delta)\delta}|\nabla h|^2.
	\end{equation}
\end{lem}

\begin{proof}
	For $a,b\geqslant0$, Young's inequality in the form
	$a^{\frac{1}{3}}b\leqslant \kappa a+\frac{2}{3\sqrt3}\kappa^{-\frac{1}{2}}b^{\frac{3}{2}}$ gives
	\[\theta F^{\frac{1}{3}}|\nabla h|^2
	\leqslant \kappa F+\frac{2}{3\sqrt3}\theta^{\frac{3}{2}}\kappa^{-\frac{1}{2}}|\nabla h|^3\]
	with $a=F$ and $b=\theta |\nabla h|^2$. Since $S$ is constant, \eqref{Sconstant-general} gives
	\[|\nabla h|^2=S(S-n)-2\sigma_2\leqslant S(S-n)\leqslant(n+\delta)\delta,\]
	which implies \eqref{eq:young}.
\end{proof}
\begin{proof}[\textbf{{Proof of Theorem~\ref{main} in codimension 2}}]
If $0\leqslant S\leqslant n$, the first-gap theorem quoted in the introduction gives exactly the two alternatives stated in Theorem~\ref{main}. Hence it remains only to rule out constant values in the interval $(n,n+\delta(n,m)]$. For this purpose it is enough to prove the stronger estimates below with the slightly larger endpoints $\delta=\frac{n}{80.04594}$ for $3\leqslant n\leqslant5$ and $\delta=\frac{n}{61.95887}$ for $n\geqslant6$, since $\frac{n}{81}<\frac{n}{80.04594}$ and $\frac{n}{62}<\frac{n}{61.95887}$. \\
Then \eqref{newlower} together with \eqref{newupper} yields
\begin{equation}\label{parame}
	\begin{aligned}
		0\leqslant\int_{M}&\bigg\lbrace\Big(S-2n-3+10\epsilon+\frac{2\epsilon^2}{n+\delta}+\sigma+\frac{2\sigma}{3}(3\mu_{1}+\max\{\mu_{1},\mu_{2}\})\epsilon\Big)|\nabla h|^2\\
		&+(3-\sigma)\sum_{\alpha,\beta}(A_{\alpha,\beta}-2B_{\alpha,\beta})-(\frac{3}{4}-\frac{\sigma}{2})F\\
		&+\frac{2\sigma}{3}\Big(\mu_{1}(S-2n)+\frac{3n-4}{2n}S-1\Big)\det\mathcal{A}\bigg\rbrace dM.
	\end{aligned}
\end{equation}
First, \eqref{sumA-2B} yields
\begin{equation}\label{3-si}
	\begin{aligned}
	&(3-\sigma)\sum_{\alpha,\beta}(A_{\alpha,\beta}-2B_{\alpha,\beta})\\
	\leqslant&(1-\frac{\sigma}{3})\left(S+4+(\Gamma_n-1)\epsilon+\sqrt[3]{c(1+\rho)F+c(1+\rho^{-1})\frac{n-2}{n}(S-\epsilon)\epsilon^2}\right)|\nabla h|^2,
\end{aligned}
\end{equation}
Using $(u+v)^{\frac{1}{3}}\leqslant u^{\frac{1}{3}}+v^{\frac{1}{3}}$ for $u,v\geqslant 0$, Lemmas~\ref{mixA-2B} and~\ref{Young}, we obtain
\begin{equation}\label{3-sig}
\begin{aligned}
&(1-\frac{\sigma}{3})\sqrt[3]{c(1+\rho)F+c(1+\rho^{-1})\frac{n-2}{n}(S-\epsilon)\epsilon^2}|\nabla h|^2\\
	\leqslant&\theta\sqrt[3]{F}|\nabla h|^2+(1-\frac{\sigma}{3})\sqrt[3]{c(1+\rho^{-1})\frac{n-2}{n}(S-\epsilon)\epsilon^2}|\nabla h|^2\\
	\leqslant&\kappa F+\frac{2}{3\sqrt3}\theta^{\frac{3}{2}}\kappa^{-\frac{1}{2}}\sqrt{(n+\delta)\delta}|\nabla h|^2+(1-\frac{\sigma}{3})\sqrt[3]{c(1+\rho^{-1})\frac{n-2}{n}(S-\epsilon)\epsilon^2}|\nabla h|^2,
\end{aligned}
\end{equation}
where $\theta$ is defined in \eqref{theta}.
Take \[x=y=\frac{\sqrt{3}}{3},\] then
\[
\mu_1=\frac{\Gamma_n}{3}
+\frac{3(n-1)+2\sqrt{n(n-2)}}{2n}\frac{\sqrt{3}}{3},
\quad
\mu_2=\frac{\Gamma_n}{3}
+\frac{3(n-1)+2\sqrt{n(n-2)}}{2n}\sqrt{3},
\]
so $\mu_2>\mu_1$.
Substituting \eqref{3-si} and \eqref{3-sig} into \eqref{parame} yields
\begin{equation}\label{finally}
	\begin{aligned}
		0\leqslant\int_M{}&
		\Big(\mathcal C_{1}|\nabla h|^2
		-\Big(\frac34-\frac\sigma2-\kappa\Big)F+\mathcal C_{2}\det\mathcal A\Big)dM,
	\end{aligned}
\end{equation}
where
\begin{align*}
	\mathcal C_{1}={}&S-2n-3+10\epsilon+\frac{2\epsilon^2}{n+\delta}+\sigma
	+\frac{2\sigma}{3}(3\mu_1+\mu_2)\epsilon
	+\left(1-\frac\sigma3\right)
	\Big(S+4+(\Gamma_n-1)\epsilon\Big)\\
	&+\frac{2}{3\sqrt3}\theta^{\frac{3}{2}}\kappa^{-\frac{1}{2}}\sqrt{(n+\delta)\delta}+\left(1-\frac\sigma3\right)
	\sqrt[3]{c(1+\rho^{-1})\frac{n-2}{n}
		(S-\epsilon)\epsilon^2},
\end{align*}
and
\[
\mathcal C_{2}=\frac{2\sigma}{3}\left(
\mu_1(S-2n)+\frac{3n-4}{2n}S-1\right).
\]
For $3\leqslant n\leqslant5$, take
\[
\begin{gathered}
	\delta=\frac n{80.04594},\qquad c=\frac{32}{15},\\
	\sigma=1.392888457230,\qquad
	\rho=0.176286456252,\qquad \kappa=0.053555771285.
\end{gathered}
\]
Then
$\frac{3}{4}-\frac{\sigma}{2}-\kappa=10^{-10}>0$. Since $S\leqslant n+\delta$, the endpoint calculation in Appendix~\ref{app:numerical-codim-two} gives $\mathcal{C}_{1}<0$ and $\mathcal{C}_{2}<0$.\\
For $n\geqslant6$, take
\[
\begin{gathered}
	\delta=\frac n{61.95887},\qquad c=\frac{24}{5}-\frac{16}{n+\delta},\\
	\sigma=1.347869050867,\qquad
	\rho=0.241247176548,\qquad \kappa=0.0760654744665.
\end{gathered}
\]
Then $\frac{3}{4}-\frac{\sigma}{2}-\kappa=10^{-10}>0$. Since $S\leqslant n+\delta$, the endpoint calculation in Appendix~\ref{app:numerical-codim-two} gives $\mathcal{C}_{1}<0$ and $\mathcal{C}_{2}<0$.

In conclusion, the coefficients of
$|\nabla h|^2$, $F$ and $\det\mathcal A$ in \eqref{finally} are strictly negative.
Thus every term in the integrand of \eqref{finally} is non-positive,
whereas the whole integral is non-negative.  Therefore the integral is zero and
all non-positive terms vanish identically. Hence,
\[
|\nabla h|^2\equiv0,
\qquad
\det\mathcal{A}\equiv0,
\qquad
F\equiv0.
\]
Finally, \eqref{SimonsS} gives
\[
0=|\nabla h|^2=S(S-n)-2\det\mathcal{A}=S(S-n).
\]
Since $S\geqslant n$, we conclude that $S\equiv n$. Therefore, the first-gap theorem \cite{GTZ} quoted in the introduction implies that, up to an ambient isometry, $M$ lies in a totally geodesic $\mathbb S^{n+1}\subset\mathbb S^{n+2}$ as the Clifford torus stated in Theorem~\ref{main}.
\end{proof}
\begin{rem}
	Further details regarding the inequalities $\mathcal{C}_{1}<0$ and $\mathcal{C}_{2}<0$ can be found in Appendix~\ref{app:numerical-codim-two}.
\end{rem}

\section{Proof of Theorem~\ref{main} in general codimension}\label{5}
In this section, we consider minimal submanifolds with flat normal bundle in general codimension. Recall from \eqref{eq:bD-def} that $b:=\sum\limits_{r=2}^{m}a_{r}=S-a_1$.  We continue to use the frame chosen in \eqref{eq:diag} in the proof below.
\begin{lem}\label{lem:u-bound}
Let $M^n$ be a closed minimal submanifold in $\mathbb{S}^{n+m}$ with flat normal bundle. If $S$ is constant, then
	\begin{equation}\label{eq:u-bound}
		\sum_{\alpha,\beta,k}\langle A^{\alpha},\nabla_{e_k}A^{\beta}\rangle^2\leqslant2b|\nabla h|^2.
	\end{equation}
\end{lem}
\begin{proof}
	 Since $S$ is constant, for each $k$ we have
	\[\langle A^{n+1},\nabla_{e_k}A^{n+1}\rangle=-\sum_{\gamma=n+2}^{n+m}\langle A^{\gamma},\nabla_{e_k}A^{\gamma}\rangle.\]
	 Hence, by the Cauchy inequality and the frame in \eqref{eq:diag},
	\[
	\langle A^{n+1},\nabla_{e_k}A^{n+1}\rangle^2=\big(\sum_{\gamma=n+2}^{n+m}\langle A^{\gamma},\nabla_{e_k}A^{\gamma}\rangle\big)^2\leqslant b\sum_{\gamma=n+2}^{n+m}|\nabla_{e_k}A^\gamma|^2.
	\]
	Also,
	\[
	2\sum_{\gamma=n+2}^{n+m}\langle A^{\gamma},\nabla_{e_k}A^{n+1}\rangle^2
	\leqslant 2b|\nabla_{e_k}A^{n+1}|^2,
	\]
	and
	\[
	\sum_{\gamma,\eta=n+2}^{n+m}\langle A^{\gamma},\nabla_{e_k}A^{\eta}\rangle^2
	\leqslant b\sum_{\eta=n+2}^{n+m}|\nabla_{e_k}A^\eta|^2.
	\]
	Finally, by Lemma~\ref{fine}, for each fixed $k$ we have
	\[
	\sum_{\alpha,\beta}\langle A^{\alpha},\nabla_{e_k}A^{\beta}\rangle^2
	=\langle A^{n+1},\nabla_{e_k}A^{n+1}\rangle^2
	+2\sum_{\gamma=n+2}^{n+m}\langle A^{\gamma},\nabla_{e_k}A^{n+1}\rangle^2
	+\sum_{\gamma,\eta=n+2}^{n+m}\langle A^{\gamma},\nabla_{e_k}A^{\eta}\rangle^2.
	\]
	Therefore, summing over $k$ proves \eqref{eq:u-bound}.
\end{proof}
In the following, we provide an upper bound for $\int_{M} |\nabla^2h|^2dM$, analogous to Theorem \ref{upperco2} in codimension 2.
\begin{thm}\label{lem:upper}
Let $M^n$ be a closed minimal submanifold in $\mathbb{S}^{n+m}$ with flat normal bundle. If $S$ is constant and $n\leqslant S\leqslant n+\delta\leqslant 2n$ for some $\delta>0$, then
	\begin{equation}\label{eq:upper}
\begin{aligned}
			\int_M|\nabla^2h|^2dM
			\leqslant
			\int_M\bigg\lbrace
			\left(S-2n-3+12\epsilon\right)|\nabla h|^2
			+3\sum_{\alpha,\beta}(A_{\alpha,\beta}-2B_{\alpha,\beta})
			\bigg\rbrace dM.
\end{aligned}
		\end{equation}
\end{thm}

\begin{proof}
	In the frame chosen in \eqref{eq:diag}, \eqref{eq:simons-gradient} becomes
	\begin{align*}
		\frac12\Delta(|\nabla h|^2)=&(2n+3)|\nabla h|^2-3\sum_{\alpha,\beta}(A_{\alpha,\beta}-2B_{\alpha,\beta})
		-6\sum_{\alpha,\beta,k}\langle A^{\alpha},\nabla_{e_k}A^{\beta}\rangle^2\\
		&-\sum_{r=1}^{m} a_r|\nabla h^{n+r}|^2+|\nabla^2h|^2.
	\end{align*}
	Using $\sum\limits_{r=1}^{m} a_r|\nabla h^{n+r}|^2\leqslant S|\nabla h|^2$ and
	Lemma~\ref{lem:u-bound}, we obtain
	\begin{equation}\label{simple1}
	\frac12\Delta(|\nabla h|^2)
	\geqslant(2n+3-S-12b)|\nabla h|^2-3\sum_{\alpha,\beta}(A_{\alpha,\beta}-2B_{\alpha,\beta})+|\nabla^2h|^2.
	\end{equation}
	Integrating \eqref{simple1} over $M$ and using $b\leqslant\epsilon$ gives \eqref{eq:upper}.
\end{proof}
\begin{lem}
Let $M^n$ be a closed minimal submanifold in $\mathbb{S}^{n+m}$ with flat normal bundle. If $S$ is constant, then
	\begin{align*}
		\frac14\Delta(S^2-|\mathcal{A}|^2)
		={}&\sum_{r=1}^m(S-a_r)|\nabla h^{n+r}|^2
		+(2n-S)\sigma_2+3\sigma_3-2\sum_{\alpha,\beta,k}\langle A^\alpha,\nabla_{e_k}A^\beta\rangle^2,
	\end{align*}
	where $\sigma_3$ is defined as in \eqref{eq:bD-def}.
	Consequently,
	\begin{equation}\label{control}
		\sum_{r=1}^m(S-a_r)|\nabla h^{n+r}|^2
		\leqslant \frac14\Delta(S^2-|\mathcal{A}|^2)-(2n-S)\sigma_2+4b|\nabla h|^2.
	\end{equation}
\end{lem}

\begin{proof}
	At any fixed point satisfying \eqref{eq:diag}, the standard computation of the Laplacian gives
	\begin{align*}
		\frac14\Delta(S^2-|\mathcal{A}|^2)
		={}&S|\nabla h|^2+nS^2-S|\mathcal{A}|^2\notag\\
		&-\Big(n|\mathcal{A}|^2-\operatorname{Tr}(\mathcal{A}^3)
		+\frac12\sum_{\alpha,\beta}|\nabla\langle A^\alpha,A^\beta\rangle|^2
		+\sum_{r=1}^{m} a_r|\nabla h^{n+r}|^2\Big).
	\end{align*}
	Using Lemma~\ref{fine}, we obtain
	\begin{equation}\label{1order}
	\begin{aligned}
		&S|\nabla h|^2-\frac12\sum_{\alpha,\beta}|\nabla\langle A^\alpha,A^\beta\rangle|^2-\sum_{r=1}^{m} a_r|\nabla h^{n+r}|^2\\
		=&\sum_{r=1}^{m}(S-a_r)|\nabla h^{n+r}|^2
	-\frac12\sum_{\alpha,\beta,k}\big(\langle \nabla_{e_k}A^\alpha,A^\beta\rangle+\langle \nabla_{e_k}A^\beta,A^\alpha\rangle\big)^2\\
	=&\sum_{r=1}^{m}(S-a_r)|\nabla h^{n+r}|^2-2\sum_{\alpha,\beta,k}\langle A^\alpha,\nabla_{e_k}A^\beta\rangle^2.
\end{aligned}
	\end{equation}
	Since $\sum\limits_{r=1}^{m} a_r^2=S^2-2\sigma_2$ and
	$\sum\limits_{r=1}^m a_r^3=S^3-3S\sigma_2+3\sigma_3$, we have
	\begin{equation}\label{0order}
	nS^2-S|\mathcal{A}|^2-n|\mathcal{A}|^2+\operatorname{Tr}(\mathcal{A}^3)=(2n-S)\sigma_2+3\sigma_3.
	\end{equation}
Summing \eqref{1order} and \eqref{0order} gives the identity
\[
\frac14\Delta(S^2-|\mathcal{A}|^2)
=\sum_{r=1}^{m}(S-a_r)|\nabla h^{n+r}|^2
-2\sum_{\alpha,\beta,k}\langle A^\alpha,\nabla_{e_k}A^\beta\rangle^2
+(2n-S)\sigma_2+3\sigma_3.
\]
Finally, using Lemma~\ref{lem:u-bound} and $\sigma_3\geqslant0$ gives \eqref{control}.
\end{proof}
Define
\[G:=\sum_{i,j,k,\alpha,\beta}\big(
4\lambda_k^\alpha\lambda_i^\beta h_{ijk}^\alpha h_{ijk}^\beta
-\lambda_i^\alpha\lambda_j^\alpha h_{iik}^\beta h_{jjk}^\beta-\lambda_i^\alpha\lambda_j^\beta h_{jjk}^\alpha h_{iik}^\beta
-2\lambda_i^\alpha\lambda_i^\beta h_{ijk}^\alpha h_{ijk}^\beta\big).\]
By \eqref{P-T type} and Lemma~\ref{fine}, $G$ can equivalently be written as
\begin{equation}\label{G}
	\begin{aligned}
		G=&-2\sum_{\alpha,\beta}(A_{\alpha,\beta}-2B_{\alpha,\beta})+2\sum_{\alpha\neq\beta}(A_{\alpha,\beta}-2B_{\alpha,\beta})\\
		&+\sum_{\alpha\neq\beta}\sum_{i,j,k}\big(4\lambda_k^\alpha\lambda_i^\beta h_{ijk}^\alpha h_{ijk}^\beta-2\lambda_i^\alpha\lambda_i^\beta h_{ijk}^\alpha h_{ijk}^\beta\big)-2\sum_{\alpha,\beta,k}\langle A^{\alpha},\nabla_{e_k}A^{\beta}\rangle^2
	\end{aligned}
\end{equation}
Therefore, substituting \eqref{G} into \eqref{1stoccur} gives
\begin{align*}
\int_{M}&3\sum_{\alpha,\beta}(A_{\alpha,\beta}-2B_{\alpha,\beta})dM\\
=\int_{M}&\bigg\lbrace3\Big(\sum_{i,\alpha,\beta,\gamma}\lambda_{i}^{\alpha}(\lambda_{i}^{\beta})^2\lambda_{i}^{\gamma}\langle A^{\alpha},A^{\gamma}\rangle-|\mathcal{A}|^2
-\sum_{\alpha,\beta,\gamma}\big(\sum_{i}\lambda_{i}^{\alpha}\lambda_{i}^{\beta}\lambda_{i}^{\gamma}\big)^2\Big)+|\mathcal{A}|^2-S^2\\
&+2\sum_{\alpha\neq\beta}(A_{\alpha,\beta}-2B_{\alpha,\beta})+\sum_{\alpha\neq\beta}\sum_{i,j,k}\big(4\lambda_k^\alpha\lambda_i^\beta h_{ijk}^\alpha h_{ijk}^\beta
-2\lambda_i^\alpha\lambda_i^\beta h_{ijk}^\alpha h_{ijk}^\beta\big)\\
&-2\sum_{\alpha,\beta,k}\langle A^{\alpha},\nabla_{e_k}A^{\beta}\rangle^2
+\sum_{\alpha,\beta,\gamma}\big(\sum_{i}\lambda_{i}^{\alpha}\lambda_{i}^{\beta}\lambda_{i}^{\gamma}\big)^2-\sum_{i,k,\alpha,\beta,\gamma}(\lambda_{k}^{\alpha})^2\lambda_{k}^{\gamma}(\lambda_{i}^{\beta})^2\lambda_{i}^{\gamma}\bigg\rbrace dM.
\end{align*}
By \eqref{0xj} and \eqref{Sconstant-general}, we obtain
\begin{align}
\int_{M}|\nabla^2 h|^2 dM
\geqslant\int_{M}&\bigg\lbrace
\frac{3}{2}\sum_{\alpha,\beta}(A_{\alpha,\beta}-2B_{\alpha,\beta})
-\frac{3}{2}|\nabla h|^2-\frac{1}{2}(|\mathcal{A}|^2-S^2)+\frac{3\left(nS-|\mathcal{A}|^2\right)^2}{2(n+4)S}\notag\\
&-\frac{1}{2}\Big(\sum_{\alpha,\beta,\gamma}
\big(\sum_{i}\lambda_{i}^{\alpha}\lambda_{i}^{\beta}\lambda_{i}^{\gamma}\big)^2-\sum_{i,k,\alpha,\beta,\gamma}(\lambda_{k}^{\alpha})^2\lambda_{k}^{\gamma}
(\lambda_{i}^{\beta})^2\lambda_{i}^{\gamma}\Big)\label{turn}\\
&-\sum_{\alpha\neq\beta}(A_{\alpha,\beta}-2B_{\alpha,\beta})+\sum_{\alpha,\beta,k}\langle A^{\alpha},\nabla_{e_k}A^{\beta}\rangle^2\notag\\
&-\sum_{\alpha\neq\beta}\sum_{i,j,k}\bigl(2\lambda_k^\alpha\lambda_i^\beta h_{ijk}^\alpha h_{ijk}^\beta-\lambda_i^\alpha\lambda_i^\beta h_{ijk}^\alpha h_{ijk}^\beta\bigr)\bigg\rbrace dM.\notag
\end{align}
Estimating the right-hand side of \eqref{turn} from below gives the following theorem, analogous to Theorem \ref{thmlower1} in codimension 2.
\begin{thm}
	Let $M^n$ be a closed minimal submanifold in $\mathbb{S}^{n+m}$ with flat normal bundle. If $S$ is constant and $n\leqslant S\leqslant n+\delta\leqslant2n$ for some $\delta>0$, then
	\begin{equation}\label{ultim}
		\begin{aligned}
			\int_M|\nabla^2 h|^2dM\geqslant\int_{M}&\bigg\lbrace\hat{\sigma}\sum_{\alpha,\beta}(A_{\alpha,\beta}-2B_{\alpha,\beta})-\big(\hat{\sigma}+\frac{8}{3}\hat{\sigma}\hat{\mu}b\big)|\nabla h|^2\\
			&-\frac{2\hat{\sigma}}{3}\Big(\hat{\mu}(S-2n)+\frac{3(3n-4)}{4n^2}(n+\delta)^2-1\Big)\sigma_2+(\frac{3}{4}-\frac{\hat{\sigma}}{2})F\bigg\rbrace dM.
		\end{aligned}
	\end{equation}
where $0<\hat{\sigma}<\frac{3}{2}$, $\hat{\mu}:=\frac{\Gamma_n}{3}+\frac{3n-3+2\sqrt{n(n-2)}}{2n}$.
\end{thm}
\begin{proof}
Using \eqref{key}, we have
\begin{equation}\label{sqrt}
	\begin{aligned}
		\sum_{\alpha\neq\beta}(A_{\alpha,\beta}-2B_{\alpha,\beta})\leqslant&\frac{\Gamma_n}{3}\sum_{\alpha\neq \beta}a_{\alpha-n}|\nabla h^{\beta}|^2\\
		=&\frac{\Gamma_n}{3}\sum_{r=1}^{m}(S-a_r)|\nabla h^{n+r}|^2.
	\end{aligned}
\end{equation}
	For $\alpha\ne\beta$, the estimates \eqref{eq:opt-two} and \eqref{eq:opt-same-index} give the following bound; when the two tangent indices coincide, \eqref{eq:opt-same-index} is used and its constant is smaller than that in \eqref{eq:opt-two}. Thus
	\begin{align*}
	\Big|2\sum_{i,j,k}\lambda_k^\alpha\lambda_i^\beta
	h_{ijk}^\alpha h_{ijk}^\beta\Big|
	\leqslant&\frac{n-1+\sqrt{n(n-2)}}{n}\sqrt{a_{\alpha-n} a_{\beta-n}}|\nabla h^\alpha||\nabla h^\beta|\\
	\leqslant& \frac{n-1+\sqrt{n(n-2)}}{2n}\Big(a_{\beta-n}|\nabla h^\alpha|^2+a_{\alpha-n}|\nabla h^\beta|^2\Big),
	\end{align*}
	and 
	\[
	\begin{aligned}
	\Big|\sum_{i,j,k}\lambda_i^\alpha\lambda_i^\beta
	h_{ijk}^\alpha h_{ijk}^\beta\Big|
	&\leqslant\frac{n-1}{2n}\sqrt{a_{\alpha-n}a_{\beta-n}}|\nabla h^\alpha||\nabla h^\beta|\\
	&\leqslant\frac{n-1}{4n}
	\Big(a_{\beta-n}|\nabla h^\alpha|^2+a_{\alpha-n}|\nabla h^\beta|^2\Big).
	\end{aligned}
	\]
	Summing the two estimates over all ordered $\alpha\ne\beta$ yields
	\begin{equation}\label{eq:R14-cross}
	\begin{aligned}
		&\sum_{\alpha\neq\beta}\sum_{i,j,k}\bigl(2\lambda_k^\alpha\lambda_i^\beta h_{ijk}^\alpha h_{ijk}^\beta
		-\lambda_i^\alpha\lambda_i^\beta h_{ijk}^\alpha h_{ijk}^\beta\bigr)\\
		\leqslant{}&\Big(\frac{n-1+\sqrt{n(n-2)}}{n}+\frac{n-1}{2n}\Big)
		\sum_{r=1}^{m}(S-a_r)|\nabla h^{n+r}|^2\\
		={}&\frac{3n-3+2\sqrt{n(n-2)}}{2n}
		\sum_{r=1}^{m}(S-a_r)|\nabla h^{n+r}|^2.
	\end{aligned}
\end{equation}
Furthermore, we observe that
\begin{equation}\label{TQ1}
\begin{aligned}
	&\sum_{\alpha,\beta,\gamma}\big(\sum_{i}\lambda_{i}^{\alpha}\lambda_{i}^{\beta}\lambda_{i}^{\gamma}\big)^2-\sum_{i,k,\alpha,\beta,\gamma}(\lambda_{k}^{\alpha})^2\lambda_{k}^{\gamma}(\lambda_{i}^{\beta})^2\lambda_{i}^{\gamma}\\
	=&\Big(\sum_{\alpha,\beta,\gamma}\big(\sum_{i}\lambda_{i}^{\alpha}\lambda_{i}^{\beta}\lambda_{i}^{\gamma}\big)^2-\sum_{\alpha}\big(\sum_{i}\lambda_{i}^{\alpha}\lambda_{i}^{\alpha}\lambda_{i}^{\alpha}\big)^2\Big)-\\
	&\Big(\sum_{i,k,\alpha,\beta,\gamma}(\lambda_{k}^{\alpha})^2\lambda_{k}^{\gamma}(\lambda_{i}^{\beta})^2\lambda_{i}^{\gamma}-\sum_{\alpha}\sum_{i,k}(\lambda_{k}^{\alpha})^2\lambda_{k}^{\alpha}(\lambda_{i}^{\alpha})^2\lambda_{i}^{\alpha}\Big).
\end{aligned}
\end{equation}
Since $b\leqslant\epsilon$ gives $a_1=S-b\geqslant S-\epsilon\geqslant n$ and $\sigma_2\geqslant a_1b$, we have $b\leqslant\dfrac{\sigma_2}{n}$.
Thus by \eqref{eq:opt-three},
\begin{align}
		&\sum_{\alpha,\beta,\gamma}\big(\sum_{i}\lambda_{i}^{\alpha}\lambda_{i}^{\beta}\lambda_{i}^{\gamma}\big)^2-\sum_{\alpha}\big(\sum_{i}\lambda_{i}^{\alpha}\lambda_{i}^{\alpha}\lambda_{i}^{\alpha}\big)^2\label{TQ2}\\
		=&\sum_{\alpha\neq\beta}\sum_{\gamma}\big(\sum_{i}\lambda_{i}^{\alpha}\lambda_{i}^{\beta}\lambda_{i}^{\gamma}\big)^2+\sum_{\alpha=\beta\neq\gamma}\big(\sum_{i}\lambda_{i}^{\alpha}\lambda_{i}^{\beta}\lambda_{i}^{\gamma}\big)^2\notag\\
		\leqslant&\sum_{\alpha\neq\beta}\sum_{\gamma}\big(\sum_{i}(\lambda_{i}^{\alpha})^2(\lambda_{i}^{\beta})^2\big)\sum_{i}(\lambda_{i}^{\gamma})^2+\sum_{\alpha=\beta\neq\gamma}\big(\sum_{i}(\lambda_{i}^{\alpha})^2(\lambda_{i}^{\gamma})^2\big)\sum_{i}(\lambda_{i}^{\beta})^2\notag\\
		\leqslant&\frac{n-2}{2n}\sum_{\alpha\neq\beta}\sum_{\gamma}a_{\alpha-n}a_{\beta-n}a_{\gamma-n}+\frac{n-2}{2n}\sum_{\alpha=\beta\neq\gamma}a_{\alpha-n}a_{\beta-n}a_{\gamma-n}\notag\\
		=&\frac{n-2}{2n}\big(S^3-\sum_{r}a_{r}^3\big)\notag\\
		\leqslant&\frac{n-2}{2n}(S^{3}-a_{1}^3)\notag\\
		=&\frac{n-2}{2n}b(S^2+Sa_{1}+a_{1}^2)\notag\\
		\leqslant&\frac{n-2}{2n}3bS^2\notag\\
		\leqslant&\frac{3(n-2)}{2n^2}(n+\delta)^2\sigma_2,\notag
\end{align}
by \eqref{use} and \eqref{TQ2},
\begin{align}
		&\Big|\sum_{i,k,\alpha,\beta,\gamma}(\lambda_{k}^{\alpha})^2\lambda_{k}^{\gamma}(\lambda_{i}^{\beta})^2\lambda_{i}^{\gamma}-\sum_{\alpha}\sum_{i,k}(\lambda_{k}^{\alpha})^2\lambda_{k}^{\alpha}(\lambda_{i}^{\alpha})^2\lambda_{i}^{\alpha}\Big|\label{TQ3}\\
		\leqslant&\sum_{\alpha\neq\beta}\sum_{\gamma}\sum_{i,k}|(\lambda_{k}^{\alpha})^2\lambda_{k}^{\gamma}(\lambda_{i}^{\beta})^2\lambda_{i}^{\gamma}|+\sum_{\alpha=\beta\neq\gamma}\sum_{i,k}|(\lambda_{k}^{\alpha})^2\lambda_{k}^{\gamma}(\lambda_{i}^{\beta})^2\lambda_{i}^{\gamma}|\notag\\
		\leqslant&\sum_{\alpha\neq\beta}\sum_{\gamma}\big(\mathop{\max}\limits_{i}|\lambda_i^{\gamma}|^2\big)\sum_{k}(\lambda_{k}^{\alpha})^2\sum_{i}(\lambda_i^{\beta})^2+\sum_{\alpha=\beta\neq\gamma}\big(\mathop{\max}\limits_{i}|\lambda_i^{\gamma}|^2\big)\sum_{k}(\lambda_{k}^{\alpha})^2\sum_{i}(\lambda_i^{\beta})^2\notag\\
		\leqslant&\frac{n-1}{n}\sum_{\alpha\neq\beta}\sum_{\gamma}a_{\alpha-n}a_{\beta-n}a_{\gamma-n}+\frac{n-1}{n}\sum_{\alpha=\beta\neq\gamma}a_{\alpha-n}a_{\beta-n}a_{\gamma-n}\notag\\
		\leqslant&\frac{3(n-1)}{n^2}(n+\delta)^2\sigma_2.\notag
\end{align}
Hence, combining \eqref{TQ1}, \eqref{TQ2} and \eqref{TQ3}, we obtain
\begin{equation}\label{TQ4}
	\sum_{\alpha,\beta,\gamma}\big(\sum_{i}\lambda_{i}^{\alpha}\lambda_{i}^{\beta}\lambda_{i}^{\gamma}\big)^2-\sum_{i,k,\alpha,\beta,\gamma}(\lambda_{k}^{\alpha})^2\lambda_{k}^{\gamma}(\lambda_{i}^{\beta})^2\lambda_{i}^{\gamma}\leqslant\frac{3(3n-4)}{2n^2}(n+\delta)^2\sigma_2.
\end{equation}
Moreover,
\[|\mathcal{A}|^2=S^2-2\sigma_2,\quad \sum_{\alpha,\beta,k}\langle A^{\alpha},\nabla_{e_k}A^{\beta}\rangle^2\geqslant0,\quad \frac{3\left(nS-|\mathcal{A}|^2\right)^2}{2(n+4)S}\geqslant0.\]
Combining these facts with \eqref{sqrt}, \eqref{eq:R14-cross} and \eqref{TQ4}, \eqref{turn} becomes
\pagebreak[3]
\begin{equation}\label{esti}
	\begin{aligned}
		\int_{M}|\nabla^2 h|^2dM\geqslant\int_{M}&\bigg\lbrace\frac{3}{2}\sum_{\alpha,\beta}(A_{\alpha,\beta}-2B_{\alpha,\beta})-\frac{3}{2}|\nabla h|^2-\Big(\frac{3(3n-4)}{4n^2}(n+\delta)^2-1\Big)\sigma_2\\
		&-\Big(\frac{\Gamma_n}{3}+\frac{3n-3+2\sqrt{n(n-2)}}{2n}\Big)\sum_{r=1}^{m}(S-a_r)|\nabla h^{n+r}|^2\bigg\rbrace dM.
	\end{aligned}
\end{equation}
Substituting \eqref{control} into \eqref{esti} gives
	\begin{equation}\label{23sigma}
	\begin{aligned}
		\int_M|\nabla^2 h|^2dM\geqslant\int_{M}&\bigg\lbrace\frac{3}{2}\sum_{\alpha,\beta}(A_{\alpha,\beta}-2B_{\alpha,\beta})-\big(\frac{3}{2}+4\hat{\mu}b\big)|\nabla h|^2\\
		&-\Big(\hat{\mu}(S-2n)+\frac{3(3n-4)}{4n^2}(n+\delta)^2-1\Big)\sigma_2\bigg\rbrace dM.
	\end{aligned}
\end{equation}
Let $0<\hat{\sigma}<\frac{3}{2}$. Taking the convex combination of \eqref{23sigma} and \eqref{0xj} with weights $\frac{2\hat{\sigma}}{3}$ and $1-\frac{2\hat{\sigma}}{3}$ proves
\eqref{ultim}.
\end{proof}
\begin{proof}[\textbf{Proof of Theorem~\ref{main} in general codimension $m\geqslant3$}]
If $0\leqslant S\leqslant n$, the first-gap theorem quoted in the introduction gives exactly the two alternatives stated in Theorem~\ref{main}. Hence it remains only to rule out constant values in the interval $(n,n+\delta(n,m)]$. For this purpose it is enough to prove the stronger estimates below with the slightly larger endpoints  $\delta=\frac{n}{86.36768}$ for $3\leqslant n\leqslant5$ and $\delta=\frac{n}{66.95630}$ for $n\geqslant6$, since $\frac{n}{87}<\frac{n}{86.36768}$ and $\frac{n}{67}<\frac{n}{66.95630}$.\\
Using $b\leqslant\epsilon$ and combining Theorem \ref{lem:upper} with \eqref{ultim}, we have
\begin{equation}\label{0}
	\begin{aligned}
		0\leqslant\int_M&\bigg\lbrace\Big(S-2n-3+12\epsilon+\hat{\sigma}+\frac{8}{3}\hat{\sigma}\hat{\mu}\epsilon\Big)|\nabla h|^2
		+(3-\hat{\sigma})\sum_{\alpha,\beta}(A_{\alpha,\beta}-2B_{\alpha,\beta})\\
		&+\frac{2\hat{\sigma}}{3}\Big(\hat{\mu}(S-2n)+\frac{3(3n-4)}{4n^2}(n+\delta)^2-1\Big)\sigma_2-\big(\frac{3}{4}-\frac{\hat{\sigma}}{2}\big)F\bigg\rbrace dM.
	\end{aligned}
\end{equation}
Finally, applying \eqref{3-si} and \eqref{3-sig} to \eqref{0}, with $\sigma$ and $\theta$ replaced by $\hat\sigma$ and $\widehat\theta$, we obtain
\begin{equation*}
	\begin{aligned}
		0\leqslant\int_{M}\bigg\lbrace\mathcal{D}_{1}|\nabla h|^2+\mathcal{D}_{2}\sigma_2-\big(\frac{3}{4}-\frac{\hat{\sigma}}{2}-\kappa\big)F\bigg\rbrace dM,
	\end{aligned}
\end{equation*}
where
\[
\widehat\theta=\Big(1-\frac{\hat\sigma}{3}\Big)\Big(c(1+\rho)\Big)^{\frac{1}{3}}
\]
and
\begin{equation}\label{D1D2-def}
\begin{aligned}
\mathcal{D}_{1}=&S-2n-3+12\epsilon+\hat{\sigma}+\frac{8}{3}\hat{\sigma}\hat{\mu}\epsilon
+\Big(1-\frac{\hat\sigma}{3}\Big)\Big(S+4+(\Gamma_n-1)\epsilon\Big)\\
&+\frac{2}{3\sqrt3}\widehat\theta^{\frac{3}{2}}\kappa^{-\frac{1}{2}}\sqrt{(n+\delta)\delta}
+\Big(1-\frac{\hat{\sigma}}{3}\Big)\sqrt[3]{c(1+\rho^{-1})\frac{n-2}{n}(S-\epsilon)\epsilon^2},\\
\mathcal{D}_{2}=&\frac{2\hat{\sigma}}{3}\Big(\hat{\mu}(S-2n)+\frac{3(3n-4)}{4n^2}(n+\delta)^2-1\Big).
\end{aligned}
\end{equation}
For $3\leqslant n\leqslant5$, take
\[
\begin{gathered}
	\delta=\frac n{86.36768},\qquad c=\frac{32}{15},\\
	\hat\sigma=1.395172650836,\qquad
	\rho=0.173359800579,\qquad \kappa=0.052413674482.
\end{gathered}
\]
For $n\geqslant6$, take
\[
\begin{gathered}
	\delta=\frac n{66.95630},\qquad c=\frac{24}{5}-\frac{16}{n+\delta},\\
	\hat\sigma=1.352606093650,\qquad
	\rho=0.236478516159,\qquad \kappa=0.073696953075.
\end{gathered}
\]
The numerical verification in Appendix~\ref{app:numerical-general} gives, in both ranges,
\[
\mathcal D_1<0,
\qquad
\mathcal D_2<0,
\qquad
\frac34-\frac{\hat\sigma}{2}-\kappa>0.
\]
Since $\sigma_2\geqslant0$, every term in the integrand above is non-positive, while the whole integral is non-negative. Hence all three terms vanish identically. In particular,
\[
|\nabla h|^2\equiv0,
\qquad
\sigma_2\equiv0,
\qquad
F\equiv0.
\]
Using \eqref{Sconstant-general}, we obtain
\[
0=|\nabla h|^2=S(S-n)-2\sigma_2=S(S-n).
\]
Because $S\geqslant n$, we conclude that $S\equiv n$. The first-gap theorem quoted in the introduction then implies that, up to an ambient isometry, $M^n$ lies in a totally geodesic $\mathbb S^{n+1}\subset\mathbb S^{n+m}$ as the Clifford torus stated in Theorem~\ref{main}.
\end{proof}
\begin{rem}
	Further details regarding the inequalities $\mathcal{D}_{1}<0$ and $\mathcal{D}_{2}<0$ can be found in Appendix~\ref{app:numerical-general}.
\end{rem}

\appendix
\section{Numerical verification for \texorpdfstring{$m\geqslant3$}{m >= 3}}\label{app:numerical-general}
For fixed $n$ and for each parameter choice below, all quantities other than $S$ are fixed at their endpoint values. From \eqref{D1D2-def}, $\mathcal D_1$ is the sum of increasing affine terms and a positive multiple of $(S-\epsilon)^\frac{1}{3}$, while $\mathcal D_2$ is increasing in $S$. Hence it is enough to evaluate both coefficients at $S=n+\delta$.

For $3\leqslant n\leqslant5$, take
\[
\begin{gathered}
	\delta=\frac n{86.36768},\qquad c=\frac{32}{15},\\
	\hat\sigma=1.395172650836,\qquad
	\rho=0.173359800579,\qquad \kappa=0.052413674482.
\end{gathered}
\]
Then
\[
\frac34-\frac{\hat\sigma}{2}-\kappa=10^{-10}>0.
\]
At $S=n+\delta$, direct substitution in \eqref{D1D2-def} gives
\[
\begin{array}{c|cc}
	n&\mathcal D_1&\mathcal D_2\\ \hline
	3&<-5.47775434\times10^{-8}&<-4.035007623\\
	4&<-1.388513273\times10^{-1}&<-5.067891427\\
	5&<-2.810756848\times10^{-1}&<-6.051982054
\end{array}
\]
Thus all three required inequalities are strict for $3\leqslant n\leqslant5$.

For $n\geqslant6$, take
\[
\begin{gathered}
	\delta=\frac n{66.95630},\qquad c=\frac{24}{5}-\frac{16}{n+\delta},\\
	\hat\sigma=1.352606093650,\qquad
	\rho=0.236478516159,\qquad \kappa=0.073696953075.
\end{gathered}
\]
Then
\[
\frac34-\frac{\hat\sigma}{2}-\kappa=10^{-10}>0,
\]
and the endpoint values satisfy
\[
\begin{array}{c|cc}
	n&\mathcal D_1&\mathcal D_2\\ \hline
	6&<-1.496434737\times10^{-2}&<-6.687029445\\
	7&<-5.962410481\times10^{-3}&<-7.591896378\\
	8&<-1.655055473\times10^{-3}&<-8.492199124\\
	9&<-1.04197246\times10^{-7}&<-9.389681336\\
	10&<-1.0419\times10^{-7}&<-10.285296449
\end{array}
\]
It remains to control the other terms in $n$. Write
\[
n+\delta=\tau n,\qquad \epsilon=\eta n,\qquad
c=c_0-\frac{c_1}{n},
\]
where $\tau,\eta,c_0,c_1$ are positive constants. After collecting affine terms, every non-affine term in $\mathcal D_1$ is a positive multiple of one of
\[
\sqrt{17n^2-26n+9},\qquad
\sqrt{n(n-2)},\qquad
\sqrt{n(c_0n-c_1)},\qquad
\big((c_0n-c_1)(n-2)n\big)^{1/3}.
\]
These four functions are concave on $[6,\infty)$; for the last one this follows from the concavity of the geometric mean on the positive cone. Hence $\mathcal D_1$ is concave and the interval computation gives
\[
\mathcal D_1(11)-\mathcal D_1(10)
<-1.105497472\times10^{-3}<0.
\]
Together with $\mathcal D_1(10)<0$, this proves $\mathcal D_1(n)<0$ for every integer $n\geqslant10$; the cases $6\leqslant n\leqslant9$ are contained in the table.

For $\mathcal D_2$, set
\[
q=1-\frac1{66.95630},\qquad
R_n=\sqrt{17n^2-26n+9},\qquad T_n=\sqrt{n(n-2)}.
\]
Differentiating the endpoint expression gives
\[
\mathcal D_2'(n)=\frac{2\hat\sigma}{3}
\left[\frac94\tau^2-q\left(\frac53+\frac{34n-26}{12R_n}+\frac{n-1}{T_n}\right)\right].
\]
Since $R_n<\sqrt{17}\,n$, $T_n<n-1$, and
$34n-26\geqslant(34-\frac{26}{6})n$ for $n\geqslant6$, we obtain
\[
\mathcal D_2'(n)
<\frac{2\hat\sigma}{3}\left[\frac94\tau^2-q\left(
\frac53+\frac{34-\frac{26}{6}}{12\sqrt{17}}+1\right)\right]
<-0.81136<0.
\]
Thus $\mathcal D_2$ is decreasing; since $\mathcal D_2(6)<0$, it is negative for every $n\geqslant6$.

\section{Numerical verification for \texorpdfstring{$m=2$}{m = 2}}\label{app:numerical-codim-two}
For fixed $n$ and for each parameter choice below, all quantities other than $S$ are fixed at their endpoint values. From the formulas following \eqref{finally}, $\mathcal C_1$ is the sum of increasing affine terms and a positive multiple of $(S-\epsilon)^\frac{1}{3}$, while $\mathcal C_2$ is increasing in $S$. Therefore it is enough to check $S=n+\delta$. 

For $3\leqslant n\leqslant5$, take
\[
\begin{gathered}
	\delta=\frac n{80.04594},\qquad c=\frac{32}{15},\\
	\sigma=1.392888457230,\qquad
	\rho=0.176286456252,\qquad \kappa=0.053555771285.
\end{gathered}
\]
Then
\[
\frac34-\frac\sigma2-\kappa=10^{-10}>0,
\]
and the endpoint values are
\[
\begin{array}{c|cc}
	n&\mathcal C_1&\mathcal C_2\\ \hline
	3&<-1.2655\times10^{-8}&<-3.401096539\\
	4&<-1.397129500\times10^{-1}&<-4.150795067\\
	5&<-2.828354554\times10^{-1}&<-4.871909089
\end{array}
\]
Thus $\mathcal C_1<0$ and $\mathcal C_2<0$ for $3\leqslant n\leqslant5$.

For $n\geqslant6$, take
\[
\begin{gathered}
	\delta=\frac n{61.95887},\qquad c=\frac{24}{5}-\frac{16}{n+\delta},\\
	\sigma=1.347869050867,\qquad
	\rho=0.241247176548,\qquad \kappa=0.0760654744665.
\end{gathered}
\]
Then
\[
\frac34-\frac\sigma2-\kappa=10^{-10}>0,
\]
and the endpoint calculation gives
\[
\begin{array}{c|cc}
	n&\mathcal C_1&\mathcal C_2\\ \hline
	6&<-1.941441181\times10^{-2}&<-5.339178672\\
	7&<-9.156050428\times10^{-3}&<-6.009723704\\
	8&<-3.720525042\times10^{-3}&<-6.677585719\\
	9&<-1.010197647\times10^{-3}&<-7.343788660\\
	10&<-3.44517493\times10^{-7}&<-8.008892855
\end{array}
\]
For the remaining dimensions, write again
\[
n+\delta=\tau n,\qquad \epsilon=\eta n,\qquad
c=c_0-\frac{c_1}{n}.
\]
After collecting affine terms, the non-affine terms in $\mathcal C_1$ are positive multiples of the same four concave functions displayed in the preceding section. Hence $\mathcal C_1$ is concave and the interval computation gives
\[
\mathcal C_1(11)-\mathcal C_1(10)
<-1.260923771\times10^{-4}<0.
\]
Since $\mathcal C_1(10)<0$, it follows that $\mathcal C_1(n)<0$ for every integer $n\geqslant10$; the cases $6\leqslant n\leqslant9$ are listed in the table.

For $\mathcal C_2$, put
\[
q=1-\frac1{61.95887},\qquad
R_n=\sqrt{17n^2-26n+9},\qquad T_n=\sqrt{n(n-2)}.
\]
At $x=y=\frac{\sqrt{3}}{3}$, differentiation of the endpoint expression yields
\[
\mathcal C_2'(n)=\frac{2\sigma}{3}\left[
\frac{3\tau}{2}-q\left(\frac16+\frac{34n-26}{12R_n}
+\frac{1}{2\sqrt3}\left(3+\frac{2(n-1)}{T_n}\right)\right)\right].
\]
Using the same elementary bounds as above, we obtain
\[
\mathcal C_2'(n)
<\frac{2\sigma}{3}\left[\frac{3\tau}{2}-q\left(
\frac16+\frac{34-\frac{26}{6}}{12\sqrt{17}}+\frac{5}{2\sqrt3}\right)\right]
<-0.58387<0.
\]
Thus $\mathcal C_2$ is decreasing on $[6,\infty)$; since $\mathcal C_2(6)<0$, it is negative for every $n\geqslant6$.

\end{document}